%% file: main.tex
 \documentclass[authoryear,3p, times,10pt,final]{elsarticle}

\include{myconfigs}

\journal{}

\begin{document}
\begin{frontmatter}



\title{Dock Assignment and Truck Scheduling Problem; Consideration of Multiple Scenarios with Resource Allocation Constraints}

\author[ITBA,Pred,CRISTAL,LILLE]{Rahimeh Neamatian Monemi}
\author[LGI2A,PSB]{Shahin Gelareh}

\corref{cor1}

\cortext[cor1]{Corresponding author, shahin.gelareh@\{univ-artois.fr;gmail.com\}; r.n.monemi@gmail.com}

\address[ITBA]{IT and Business Analytics Ltd}
\address[Pred]{Sharkey Predictim Globe}
\address[CRISTAL]{CNRS UMR 9189, Centre de Recherche en Informatique, Signal et Automatique de Lille}
\address[LILLE]{Université de Lille, France}
\address[LGI2A]{Département R\&T, IUT de Béthune, Université d'Artois, F-62000 Béthune, France}
\address[PSB]{Paris School of Business, Paris, France}

\begin{abstract}
The notion of ’resource’ plays an important role in the overall efficiency and performance of most cross-docks. The processing time can often be described in terms of the resources allocated to different trucks. Conversely, for a given processing time, different combinations of resources can be prescribed. We study the problem of truck scheduling and dock assignment in the presence of resource constraints. In the absence of a closed-form (or well-defined) linear formulation describing the processing times as a function of resources, expert’ knowledge has been mobilised to enable modelling of the problem as an integer linear model. Two cases are taken into account: In the first one, the expert believes in his/her estimation of the processing time for every truck and only proposes a different combination of resources for his/her estimation, while in the second one the expert proposes a limited number of resource deployment scenarios for serving trucks, each of which has a different combination of resources and different processing times. We propose a novel compact integer programming formulation for the problem, which is particularly designed with an embedded structure that can be exploited in dual decomposition techniques with a remarkably computationally efficient column generation approach in this case. The case in which a scenario with invariant processing time is considered and modelled as a special case of the proposed model. Since a direct application of commercial solvers such as CPLEX to solve instances of this problem is not realistic, we propose a branch-and-price framework and, moreover, several classes of valid inequalities. Our extensive computational experiments confirm that the proposed exact solution framework is very efficient and viable in solving real-size instances of the practice and in a reasonable amount of time.
\end{abstract}


\end{frontmatter}
%
\input{Intro}
\input{Literature}

\input{problemdescription}

\input{mathematicalmodel}
\input{solutionMethod}

\input{numerical}

\input{conclusion}

%



\bibliographystyle{elsarticle-harv}
\bibliography{introBib1,petroleum,crossdock,MIP,vns,alns, simulation}

\end{document}

%% file: myconfigs.tex
\usepackage{amsfonts}
\usepackage[htt]{hyphenat}

\usepackage{eulervm}
\usepackage[mathscr, mathcal]{euscript}
\usepackage{anyfontsize}
\usepackage{float}
\restylefloat{figure}
\restylefloat{table}
\usepackage{graphicx}
\usepackage{caption}
\usepackage{subcaption}
\usepackage{amsmath}

\usepackage{amsmath}

\usepackage[normalem]{ulem}

\usepackage{float}
 \usepackage{graphics}
\usepackage[colorlinks=true,linkcolor=blue, pdfstartview=FitH, breaklinks=true]{hyperref}
\usepackage{longtable,lscape,verbatim}
\usepackage{amsmath}
\usepackage{graphicx}
\usepackage[font=small,labelfont=bf]{caption}
\usepackage{latexsym}
\usepackage{amsmath}
\usepackage{todonotes}
\usepackage[nomargin,inline,marginclue,draft]{fixme}
\makeatletter
\renewcommand*\FXLayoutInline[3]{%
	{\@fxuseface{inline}\ignorespaces[#3 \fxnotename{#1}: #2]}}
\makeatother
\FXRegisterAuthor{sv}{asv}{\color{red}Me}
\allowdisplaybreaks
\graphicspath{{figs/}}
\usepackage{verbatim}
\usepackage{amsmath,amsthm}
\allowdisplaybreaks
\usepackage{algorithm}
\usepackage{algpseudocode}
\usepackage{fancybox}
 \usepackage{times}
\usepackage{fixme}

\newtheorem{definition}{Definition}

\newtheorem{corollary}{Corollary}
\newtheorem{thm}{Theorem}
\newtheorem{lem}[thm]{Lemma}

\usepackage{amsmath}
\usepackage{etoolbox}
\AtBeginEnvironment{align}{\setcounter{subeqn}{0}}
\newcounter{subeqn} %

\usepackage{listings}
\usepackage{charter}
\usepackage[expert]{mathdesign}
\usepackage{pgfplots}
\usepackage{times} 
\usepackage[compatibility=false]{caption}

\usepackage{graphicx}
\usepackage{caption}
\usepackage{subcaption}

\usepackage{filecontents}
\usepackage{pgf, tikz}
\usetikzlibrary{arrows, automata}

\usepackage{lipsum}

\allowdisplaybreaks[1] 

\tikzstyle{vertex}=[circle,fill=black!25,minimum size=20pt,inner sep=0pt]
\tikzstyle{selected vertex} = [vertex, fill=red!24]
\tikzstyle{blue vertex} = [vertex, fill=blue!24]
\tikzstyle{edge} = [draw,thick,-]
\tikzstyle{weight} = [font=\small]
\tikzstyle{selected edge} = [draw,line width=5pt,-,red!50]
\tikzstyle{ignored edge} = [draw,line width=5pt,-,black!20]

\usepackage[pdf]{pstricks}%
\usepackage{pstricks-add, pst-grad}
\definecolor[ps]{bordercolor}{rgb}{0.06 0.42 0.11}
\definecolor[ps]{greenbegin}{rgb}{0.47 0.94  0.151}
\definecolor[ps]{greenend}{rgb}{0.23 0.65  0.02}

\usepackage{setspace}
\allowdisplaybreaks

\usepackage{todonotes}

\tikzset{inlinenotestyle/.append style={align=justify}}

\usepackage{lipsum} 

\usepackage{color}

\usepackage{xcolor}

\usepackage{dsfont}
\usepackage{longtable}
\usepackage{lscape}
\usepackage{ctable}

\DeclareMathAlphabet{\pazocal}{OMS}{zplm}{m}{n}

\usepackage[utf8]{inputenc}

\usepackage{cancel}
\makeatletter
\newcommand{\crossout}[1]{%
  \begingroup
  \settowidth{\dimen@}{#1}%
  \setlength{\unitlength}{0.05\dimen@}%
  \settoheight{\dimen@}{#1}%
  \count@=\dimen@
  \divide\count@ by \unitlength
  \begin{picture}(0,0)
  \put(0,0){\line(20,\count@){20}}
  \put(0,\count@){\line(20,-\count@){20}}
  \end{picture}%
  #1%
  \endgroup
}
\makeatother

\usepackage{lscape}
\DeclareGraphicsExtensions{.png,.eps}

\usepackage{pdflscape}

\usepackage{tikz-network}
\usetikzlibrary{shapes}

\usepackage{xcolor,cancel}

\usepackage{ctable} 

%% file: intro.tex
\section{Introduction}
{}A cross-dock is a hub node in a supply chain network, wherein Less Than Truck Load (LTL) shipments of different origins are unloaded, sorted, reloaded and sent to different destinations, aiming at reducing the transportation cost as a result of exploiting economies of scale through higher utilisation of transporters. While the use of cross-docking operations dates back at least to the 1950s by the US army and later in the 1980s by Wal-Mart, the concept has become much more involved in the presence of cross-dock networks, specially in recent years. This trend is expected to continue in the years to come. Therefore, it is crucial that such facilities perform optimally for the overall supply chain management also to perform closer to optimally and to remain competitive. In doing so, the need for combinations of fundamental/scaffold models and very efficient solution algorithms ‒which can easily be tailored and integrated within the decision support systems‒ is more vital than ever.\\

\black{}In the truck scheduling and dock assignment problem in cross-docks, the resources available at the facility play a crucial role in the overall performance of the system. Yet, given the current state-of-the-art of hardware technologies, incorporating such realistic features causes additional complexity and significant computational performance issues when general-purpose standard solvers are employed. As a result, in the literature, one witnesses a tendency towards developing heuristic approaches to deal with more realistic models (see \cite{gu2007research}, \cite{Boysen2010413}, \cite{van2012cross} and \cite{theophilus2019truck}). However, one of the drawbacks of such heuristic methods is that often they do not deliver any indication of the solution quality. In brief, for finding the optimal solution of instances of realistic size in the literature, we often have to trade inclusion of more practical features of real-life situations (which result in intractable models to solve) in favour of having more computationally-friendly models.\\

In what concerns this study and many other similar practical cases, the operational resources (excluding IT and ICT resources) within a cross-dock are often divided into three major groups: a) The human resources (skilled or unskilled workers), b) Equipment (e.g. conveyors, which depend on the type of operations to be performed within the facility), and c) Transporters (e.g. lift-trucks, etc.). The facility operators exploit their expert’ knowledge to determine the required resources to be deployed and the overall processing time required for every truck. It is clear that fewer resources may lead to a higher processing time beyond a certain limit (not always proportionally). In practice, among all the scenarios corresponding to the possible combinations of such resources, only a few are practically feasible from an expert viewpoint. This viewpoint is often explained by the existence of very complex and informal constraints (including human resource management, subcontractors’ mode of operation and their reluctance towards certain conditions and several other informal reasons), in the sense that it leaves only a few and rather conservative options on the table. From this point of view, it is hard to accept the introduction of many unfamiliar scenarios, which are often perceived as potentially troublesome ones. As such, in this work we also rely on expert’s’ knowledge to input the set of potential combinations of resources to carry out every task (such as serving trucks) and their corresponding estimated processing time. Therefore, the set of possible scenarios is no longer an endogenous part of the problem.\\

Over the past five years, we have carried out a dozen consultancy projects for small to medium size cross-docks. We have two main observations: 1) \textsc{Multiple scenarios with the same processing time}: based on some repeating patterns when dealing with a given client, operators have a highly accurate estimation of the processing time $\Delta$, required to serve a given truck of a given client (even sometimes as a function of the day of the week), in which case they are capable of proposing different \emph{scenarios} or, in other words, some other combinations of resources lead to almost the same processing time $\Delta$. These operators are often dealing with some informal constraints (e.g. human resource aspects, issues concerning subcontractors and some other operational constraints) which push them to remain very conservative, 2) \textsc{Multiple scenarios with varying processing times}: operators are aware of the repeating patterns of dealing with trucks of a given client but at the same time are less conservative (perhaps due to having better control over the entire in-house logistic operation, have made recent investment  in equipment, making them less tight in operational capacity and more ambitious to increase their throughput or for any other reason) and are more tolerant towards experimenting with a limited number of various (sometimes similar) processing times corresponding to the distinct resource allocation scenario. In reality, again, they are only willing to work with a very small set of possible scenarios.
\\

It needs to be emphasised that, in reality, we do not have an infinite number of scenarios with decimal point precision in their processing time. The level of granularity introduces some rather larger time segments with discrete values, so a significant reduction in the cardinality of the set of possible processing times is of help. One issue remains though; every such processing time often corresponds to more than one scenario, causing some level of symmetry in the mathematical programming context.\\


Within a cross-dock (see \autoref{fig::internal}), resources are allocated to each gate, which mainly include lift-trucks and drivers but also some other less technical staff. In a pure transshipment cross-dock (\emph{0-inventory}) such as the one depicted in \autoref{fig::internal}, no stacking and inventory is taking place. However, in practice some of the very modern cross-docks are also equipped with intelligent racks and automated stacking systems.\\


\begin{figure}
        \begin{subfigure}[b]{0.48\textwidth}
                \includegraphics[width=\linewidth]{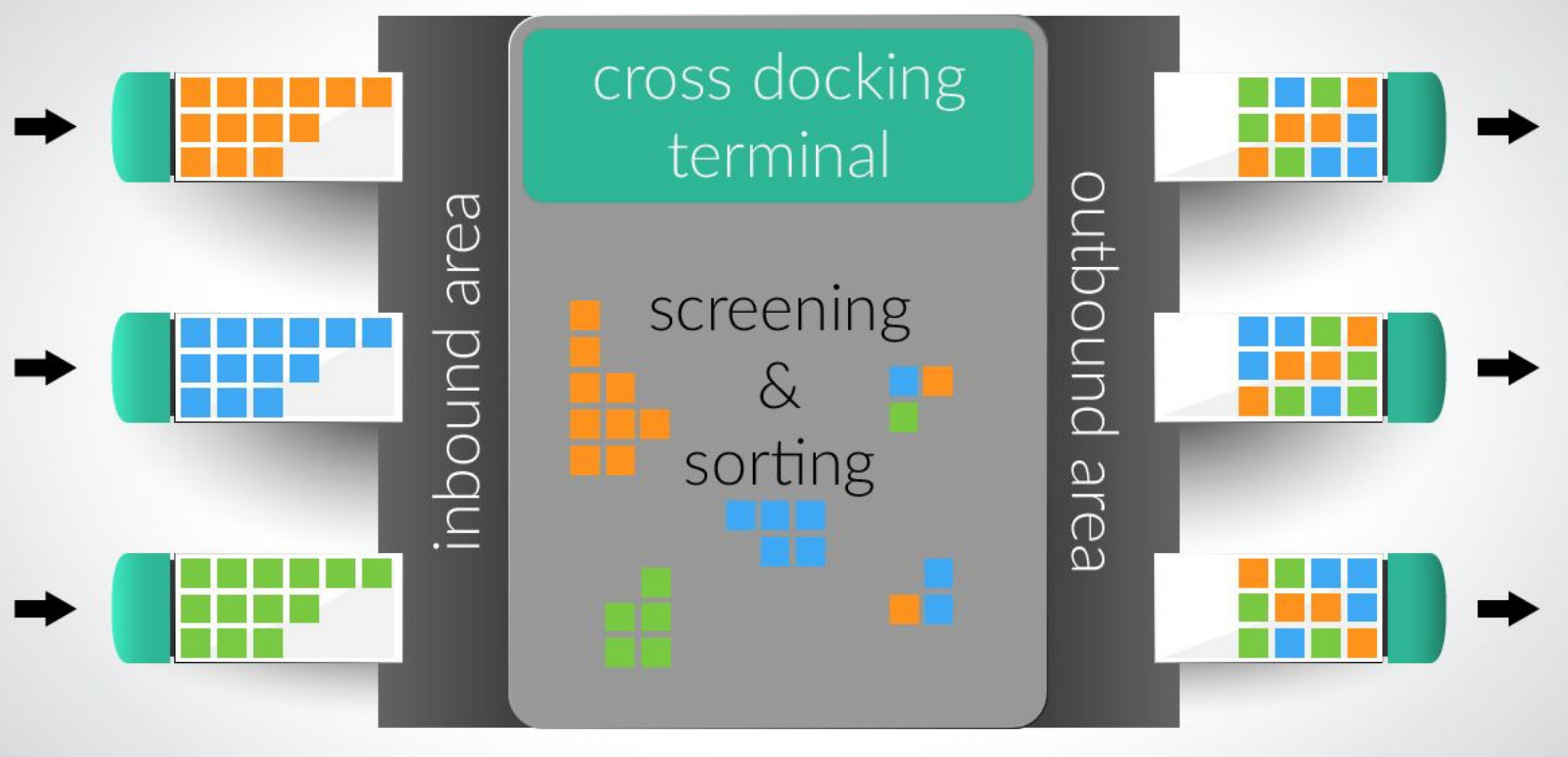}
                \caption{Cross-docking operations \footnote[1]{(source: https://www.odoo.com/)}.}\label{fig::plan} 
        \end{subfigure}%
        \begin{subfigure}[b]{0.48\textwidth}
                \includegraphics[width=\linewidth]{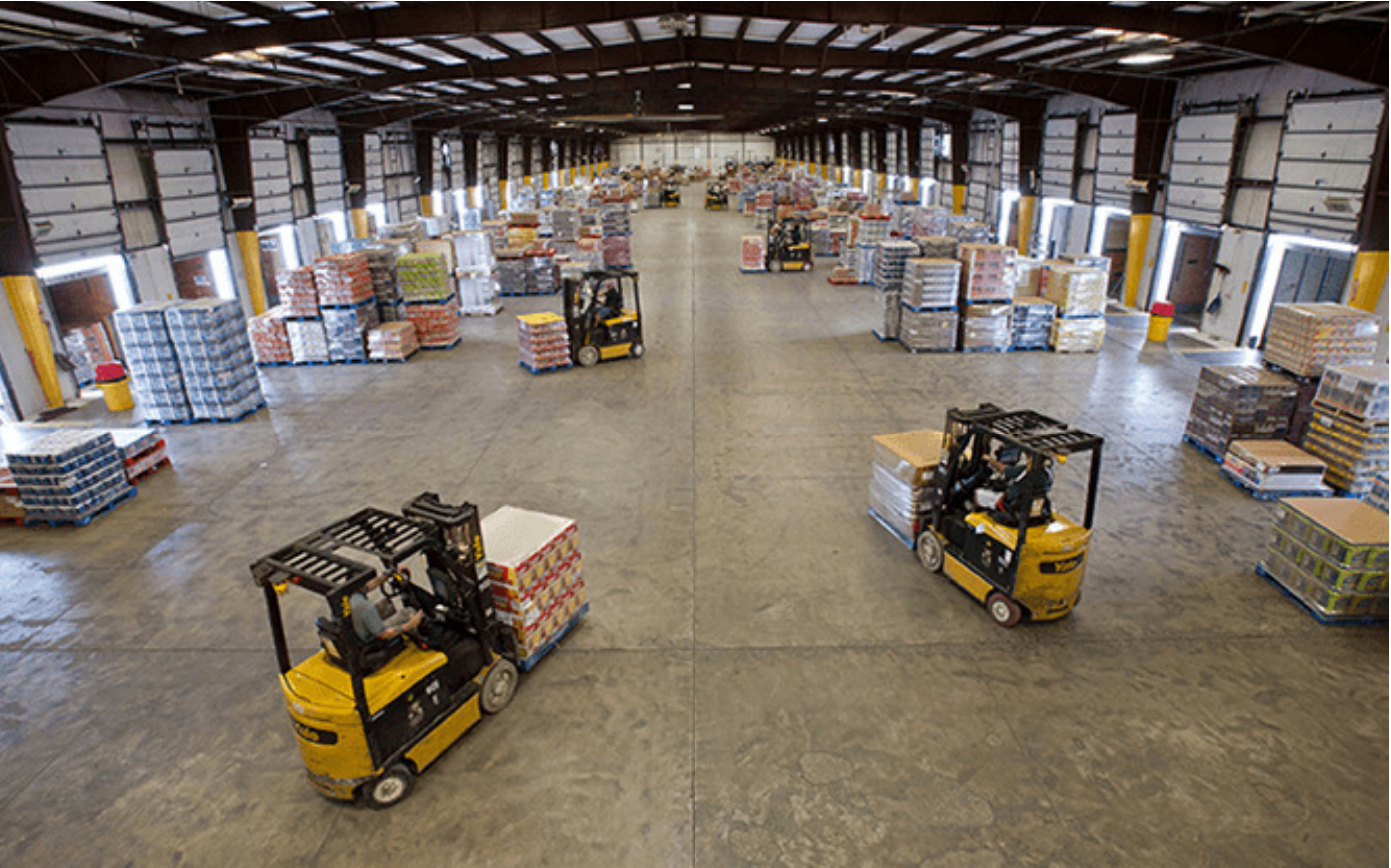}
                \caption{Internal operation of a cross-dock \footnote[2]{(https://www.sandktransport.com)}.}\label{fig::internal} 
        \end{subfigure}%
        \caption{A typical cross-dock.}\label{fig:cross-dock}
\end{figure}

Once truck operations are completed, the loaded trucks head towards different destinations (ports, other cross-docks, distribution to clients, etc.). To be able to make it to the other destinations, trucks often have less flexibility in their departure times, in particular those visiting other consolidation/distribution facilities within the same day or the day after. If their deadlines are not met, they will miss their appointments at the other facilities and this may propagate across the whole or part of the supply chain (even if the entire chain belongs to the same stakeholder). Therefore, in order to maintain the service level, and avoid the domino effect of delay propagation and client dissatisfaction, cross-docks are bound to respect the time window of every truck.\\
\\

This work has been motivated by a series of consultancy projects we have been carrying out for a few years now, in a regional multimodal logistic platform. The type of cross-dock we are dealing with is akin to the one in \autoref{fig::internal}. The pallets unloaded from a truck (or those to be loaded to another truck) do not have to be transferred immediately and can remain in the buffer zone for a few hours during the day. This makes it unnecessary to plan for a direct transfer and for an additional synchronization to ensure the simultaneous presence of unloading and loading trucks at the corresponding gates.

\subsection{Contribution and scope}
In spite all the efforts and advances in (mixed) integer programming formulations proposed for variants of the truck scheduling and dock assignment problem in cross-docks, still solving to optimality anything even near a real-size instance remains a challenge even for the models with the least features of the real practice, let alone incorporating resources and further realistic features. \
In practice, as far as we have learned from working with quite a few mid-size cross-dock operators, the operators work in time slots of a few minutes, say $\tau$ minutes, rather than treating time as a continuous measure. Every working shift is composed of an exact number $\alpha = \frac{8~ hours \times 60~ minutes}{\tau}$ minutes of time slots and for serving every truck an integer quantity of consecutive time slots is allocated.

Our contribution to the literature and practice is manifold: 1) We present a novel problem description from the viewpoint of practitioners. Instead of unnecessarily increasing the search space of feasible solutions  by integrating further excessive combinatorial components into the formulation (concerning resources and the dependency of processing time on the quantity of each resource in every possible combination and thus, having the symmetry issue as a result) or by components with unknown closed-form formulation, we show how expert’ knowledge can help to incorporate a controlled-size combinatorial component into the problem description to accommodate the realistic features. 2) We propose {a} novel integer programming model for the problem. Our proposed model is designed by discretising the time dimension of the problem and embedding decomposable structures aiming at exploiting these structures in designing solution approaches and solving practical size instances of the problem to optimality. We are unaware of any similar/comparable model in the literature.\black{} 3) By exploiting the embedded structure, we propose for each integer programming formulation, a decomposition scheme based on a Dantzig‒Wolfe reformulation and solve this decomposition within a branch-and-price framework. Our extensive computational experiments confirm that the proposed solution approach is indeed very efficient and delivers practice-friendly optimal solutions for the instances in a reasonable time. 4) Constrained by some confidentiality restrictions, rather than relying on a purely random test-bed, we extract the distribution of different events (arrival, departures, service time, etc.) from real data and use a discrete-event simulation model to simulate and generate our test-bed. 5) The model, as it is, can be used to accommodate still further features of the real practice. \black{}6) The quality of our reported optimal solutions has been confirmed mathematically and practically, by achieving feasible solutions resulting in a cost reduction of between $5.3$ and $11.2 \%$ compared to the current practices.
\\

The paper is structured as follows. In section \ref{sec:literature} we review the relevant literature. In section \ref{sec::problemdescription}, a clear description of the problem is presented. In section \ref{sec::model}, we introduce our mathematical model. \black{}Section \ref{sec::solution} elaborates our exact solution approaches and section \ref{sec::numerical} reports some computational results and analysis. Section \ref{sec::conclusion} summarises the work, presents a conclusion and sheds more light on further research directions.

%% file: literature.tex
\section{Literature review}\label{sec:literature}

The combined dock assignment and truck scheduling problem, as many other optimization problems arising in cross-dock management, has been receiving an increasing attention during the last decade.
This is evidenced by several recent surveys (\cite{gu2007research}, \cite{Boysen2010413}, \cite{van2012cross}) and more recently, \cite{theophilus2019truck} devoted to the operations management and scheduling in cross-dock.\\

This work belongs to the family of problem descriptions dealing with the assignment of docks to the trucks. Yet, as our cross-docks do include buffer zones, our problem is not a pure transhipment problem and we do not distinguish between inbound and outbound docks. To the best of knowledge, this problem in the way described in this work, has not been visited before in the literature. Yet, the literature is aware of many fundamental models and sophisticated solution methods proposed for problems in this context. In what follows, we only cite a few of those works that are somehow related to our one.\\

 \cite{berghman2011optimal} presented a model for the dock assignment problem, assuming that the trailers need to be assigned to the gates for a given period of time for loading/unloading activities. In this study, the parking lot is used as a buffer zone, while transportation between the parking lot and the gates is performed by tractors owned by the terminal operator. The problem is modeled as a three-stage flexible flow shop problem, where the first and the third stage share the same identical parallel machines (for unloading and loading activities) and the second stage consists of a different set of identical parallel machines (buffer processing activities). They examined several integer programming formulations for the proposed problem.\

Some of the work dealing with door/dock assignment problem with or without scheduling of the truck sequences are elaborated in the following. One observes that those works that deal with solution approach, focus mainly on the meta-heuristic-like approaches as problems become computationally very intractable for any commercial MIP solver.\\

The seminal work of \cite{tsui1992optimal} proposes a bilinear model for the truck dock assignment problem for a shipping company. \cite{lim2005transshipment} studied the transshipment problem with supplier and customer time windows where the flow is constrained by transportation schedules and warehouse capacities. The objective is to minimize the total cost, including the inventory costs. \cite{lim2006truck} proposes a tabu search as well as a CPLEX-based GA for the same problem as in \cite{lim2005transshipment}. \cite{deshpande2007simulating} introduce a dock assignment heuristic. They integrated the tactical level decision-making process and the operational aspects in LTL terminals to evaluate the performance of the system. 
\cite{Miao2009105} propose an integer programming formulation for the truck dock assignment problem in which the capacity of a cross-dock is explicitly taken into account over the planning horizon. Two meta-heuristics based on tabu search and  genetic algorithm are proposed to solve this problem.
\cite{cohen2009trailer} discuss the existing approaches and proposed a formulation and a new heuristic for assigning cross-dock doors to trailers.   \cite{Yu2008377} seek the best truck docking or scheduling sequence for both inbound and outbound trucks to minimize the total operation time when a temporary storage buffer is located to hold items temporarily at the shipping dock. In this work, the product assignment to the trucks and the docking sequences of the inbound and outbound trucks are determined simultaneously. \cite{boloori2011meta} proposes meta-heuristics to find the best sequence of inbound and outbound trucks, so that the objective, which is minimizing the total operation time called makespan, can be satisfied. \cite{arabani2010multi} deal with scheduling problem of inbound and outbound trailers in a cross-docking system according to the just-in-time approach. \cite{liao2012simultaneous} study the problem of determining simultaneously the dock assignment and the sequence of inbound trucks. They consider a multi-door cross-docking operation with the objective to minimize  the total weighted tardiness, under a fixed outbound truck departure schedule. The problem is solved by six different metaheuristic algorithms, which include simulated annealing, tabu search, ant colony optimization, differential evolution, and two hybrid differential-evolution algorithms.
\cite{fathollahi2019novel} consider a truck scheduling problem at cross-docks and propose a so called  Social Engineering Optimizer (SEO). Thye experiment on a  real case study to see how the algorithm performs in a real-life situation. \cite{shahmardan2020truck} takes into account the situation in which trucks can be partially unloaded and loaded at the same time and separating inbound and outbound docks are not required. They formulated the problem as a MIP aiming at finding the dock assignment as well as truck scheduling to minimize the makespan. A heuristic-based solution approach including several sophisticated neighborhood structures, equipped with a learning mechanism is proposed. \cite{dulebenets2021adaptive} proposes a so called Adaptive Polyploid Memetic Algorithm (APMA) for the problem of scheduling trucks. Their main contribution is to propose a solution method, which reports promising results. The work presented in \cite{theophilus2021truck} is concerned with the cold supply chain and perishable goods and cross-docks that deal with this kind of products. They proposed a MIP model for the truck scheduling optimization problem  by incorporating features such as decay of perishable products throughout the service of arriving trucks as well as the presence of temperature-controlled storage areas aiming at optimizing  the total cost incurred during the truck service. They also proposed an evolutionary approach to solve instances of this problem.\\
  \\

Several works tried to come up with new formulations but not necessarily with the aim of overcoming the computational barriers or proposing exact solution methods. \cite{shakeri2008generic} addressed the two problems of truck scheduling and truck-to-door assignment jointly in a mixed integer programming model. \cite{Miao2014} proposed a different model (compared to the one in \cite{Miao2009105}) and an adaptive tabu search algorithm to address the un-capacitated case of the problem.\\

To the best of our knowledge, \cite{GELAREH20161144} is the only work came up with a new formulation and carried out a through polyhedral analysis of the truck-dock assignment problem. Based on those findings, they proposed a new model, identified the dimension of the polytop associated to their model, introduced several classes of valid inequalities, proved that some classes are facet-defining and proposed a branch-and-cut approach pushing up the size of problems that could be solved to optimality. Later in \cite{GELAREH2020102015}, they proposed a set of 11 integer programming formulations and carried out an extensive comparison study on different integer programming models.\\
 	
 \cite{ou2010scheduling} dealt with air cargo and formulated the problem as a time-indexed integer programming problem. They showed that even with a limited number of unloading docks at a terminal, the problem is still quite difficult to solve. They also proposed an exact solution procedure to determine an optimal unloading sequence for the shipments carried by each truck, together with a Lagrangian relaxation-based heuristic for assigning the trucks to the docks and determining the trucks arrival time. Other similar models with application in airport gate assignment have been reported in \cite{babic1984aircraft,ding2004new,ding2005over,oh2006dock} and \cite{lim2005airport}. \\

An approach based  on the discrete-event simulation including an enhanced robust design technique was proposed by \cite{Shi2013695}. The authors addressed a multi-response optimization problem inherent in logistics management. They aimed at designing a robust configuration for a cross-docking distribution center so that the system is insensitive to the disturbances of supply uncertainty and provides steady parts supply to downstream assembly plants.\\

The concept of cross-dock networks has also become more popular in recent years. \\

\cite{agustina2010review} addressed the distribution planning problem of cross-docking networks, considering transshipment possibility among different cross-docks and tardiness permission for some pickups. The problem is formulated as a bi-objective integer programming model
minimizing the total transportation, holding costs and the total tardiness. A heuristic procedure to construct an initial solution and three metaheuristics are proposed. \cite{kim2013gate} considered the problem of minimizing the total cost of a multi-cross-dock distribution network, including transportation cost, inventory handling cost and penalty cost. They proposed an adaptive tabu search and an adaptive genetic algorithm to solve instances of the problem efficiently.\\

\cite{Buijs2014593} proposed a framework specifying the interdependencies between different aspects of the cross-docking problem and a new general classification scheme for the cross-docking research based on the inputs and outputs for each problem aspect. They also highlighted the importance of synchronization in cross-docking networks and described two real-life illustrative problems.\\

\cite{wisittipanich2019truck} presented a MIP model of truck scheduling problem in the cross-docking network comprised of multiple cross-dock and emphasized on the importance of synchronisation within this network. They sought to find the optimal truck schedule and product transshipment aiming at minimizing the makespan. They report promising results. \\

In contrast to many other work in the literature, \cite{tadumadze2019integrated} did not assume a fixed un/loading time and emphasized that in real-world, terminal managers have the additional flexibility of adapting the workforces for processing critical trucks. Their proposed MIP model integrates human resources and truck scheduling. Their computational experiments are based on proposed heuristic approach. \cite{castellucci2021network} focused on a corss-dock network. They propose a MIP model aiming at optimizing the distribution and delay costs for the transportation of goods in open networks with multiple cross-dock. They proposed a  logic-based Benders decomposition strategy which allow for the solution of larger instances when compared with those that can be handled by the direct application of a general-purpose MIP solver.\\

%% file: problemdescription.tex
\section{Problem Description}\label{sec::problemdescription}

The problem is described as follows:

\emph{Let { $\mathcal{J} = \{1, . . . , N\}$} represent the set of trucks (to be loaded  or unloaded), $\mathcal{D}=\{1,\dots,D\}$ represent the set of bi-functional dock doors (loading and unloading), and $\mathcal{R}^\dag=\{1,\dots, R^\dag\},~ ^\dag\in\{p,e,v\}$ represent the resources ($p$:personnel, $e$:equipment and $v$:vehicles) within a cross-dock and all are given. Expert beliefs and the historical data reveal that for every truck $j\in\mathcal{J}$, there is a set of resource deployment scenarios $\mathcal{S}^j=\{1,\dots,S\}$, out of which at most one will be adopted to plan serving a given truck. Let us assume that every truck $j\in \mathcal{J}$ has a given arrival time, $r_j$, and a strict latest departure time, $d_j$, a docking (setup) time $\delta_j$ (which is the time required for aligning the truck in front of a gate, the setups required by both the truck and the dock to start the loading/unloading operation) and a processing time $p_{j}^{s}$ (units of time interval) required to load/unload a truck following a resource deployment scenario $s$, and these are all given. In its general form, every truck $j$ may belong to different clients and carries several cargos with a variety of sensitivities translated to a penalty cost, $f_j$, for every unit of waiting time before getting admitted to service. A penalty cost $g_j$ is charged when truck $j$ is not served at all during the planning horizon (on the official planning of the day). We want to minimise: 1) the waiting cost of every single truck such that it can leave the soonest possible, and 2) the weighted penalty cost of missed trucks in the daily planning of the cross-dock. Therefore, the objective function accounts for the sum of penalty costs associated with the waiting times of all trucks before being admitted to be served and a more significant penalty cost associated with the trucks not being served.}

It is assumed that the cross-dock has an overly large capacity, making it practically un-capacitated.

\input{graph} 

%% file: graph.tex
\subsection{Graphical representation of the problem}
A visual representation of our modelling framework as a graph is given in \autoref{fig:graphical}. The modelling is composed of a ’dummy’ truck $0$ as the start and another ’\emph{dummy}’ truck $0$ as the end truck that every dock has to serve.\\

\begin{figure}
  \centering
  \includegraphics[width=0.8\textwidth]{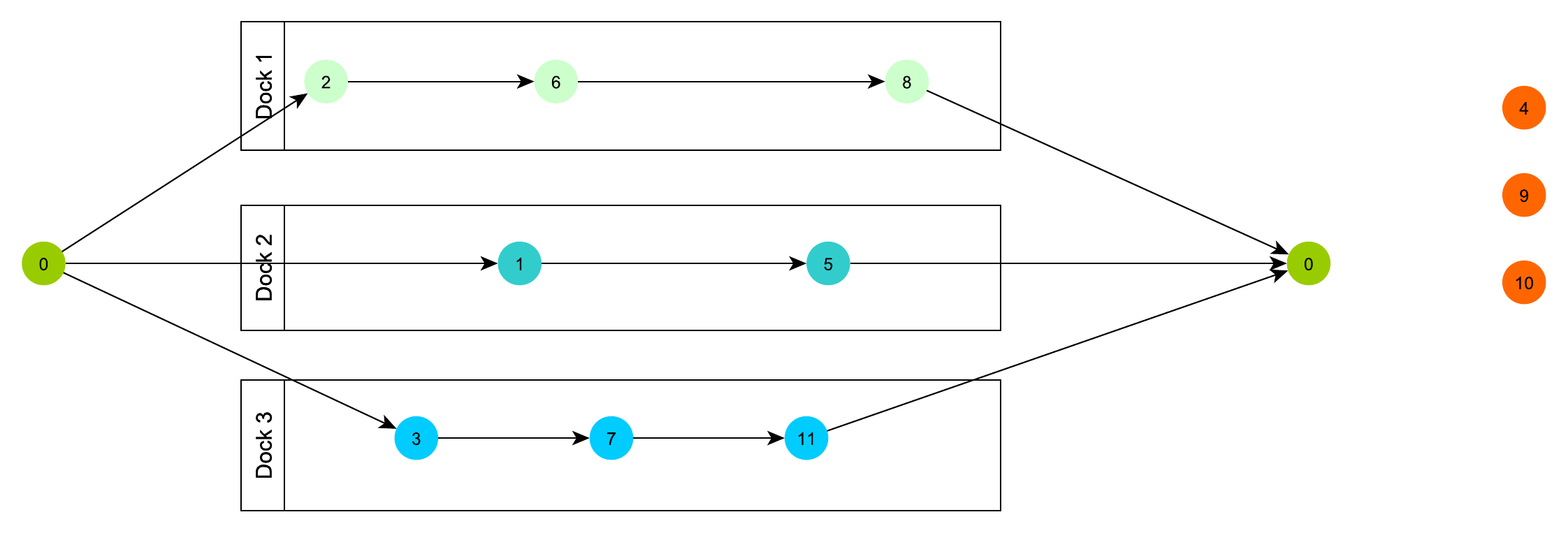}
  \caption{A graphical representation of the modeling approach.}\label{fig:graphical}
\end{figure}
Every truck is served by no more than one dock door and each horizontal box represents a dock door. Each dock door has a list of trucks to serve and the horizontal axis stands for the time axis. The nodes stacked next to the graph (on the right side), namely 4, 9 and 10, are the ones that are not served by any dock door. The total number of arcs in this graph minus the number of dock doors is equal to the number of trucks being served. In addition, while every real truck is preceded by one arc and is followed by another arc, the first and the last nodes (dummy ones) encompass and receive a number of arcs equal to the number of dock doors.\\

If no real truck is being served by a dock door, then there will be an arc connecting the first dummy to the last dummy node for every such idle dock door. 

%% file: mathematicalmodel.tex
\section{Mathematical Model}\label{sec::model}
We propose a model for the general case of \emph{Dock Assignment and Truck Scheduling Proble (DATS)} wherein one assumes scenario-dependent service times for serving a truck, and explain how the other case --a scenario-invariant one wherein all identified scenarios for every truck have the same processing time‒- can be derived. Indeed, we adopt an arc-time indexed formulation approach.\\

\subsection {Mathematical Model for the Multiple Scenario DATS Problem with Resource Allocation} \label{subsec:params}

The following parameters are used (see \autoref{tab:params}): For the purpose of modelling, we define a set of trucks, $\mathcal{J}$, a dummy truck '0' and $\mathcal{J}^\circ = \mathcal{J}\cup \{0\}$, while assuming that all docks start the day by \emph{completing} the service to the dummy truck '0' with zero processing time, and terminate it by \emph{starting} to serve the dummy truck '0' with zero processing time. The planning horizon is divided into discrete time segments $\mathcal{T}=\{0,1,2,\dots, T\}$. We assume that $a_0=\delta_0 = 0$, $d_0=T$. Obviously, $|\mathcal{D}|<< |\mathcal{J}|$. Furthermore, it is reasonable to assume that the number of actually served trucks in every instance is no less than the number of available docks $|\mathcal{D}|$. 
\\

{
\begin{table}[h]
\caption{Parameters.}
\label{tab:params}
\centering
\scalebox{0.8}{
\begin{tabular}{l l}
\toprule

$S=\bigcup_{j\in\mathcal{J}} \mathcal{S}^j$ & set of all scenarios,\\
{}$\mathcal{D}=\{1,\dots,D\}$ &    set of dock doors,\\
'0' & dummy truck, \\
$\mathcal{J}$& set of trucks, $\mathcal{J}^\circ = \mathcal{J}\cup \{0\}$,\\
$\mathcal{T} = \{0,1,2,\dots, T\}$ & planning horizon, \\
{}$p_j^s$ & processing time, the time required to serve truck $j$ under scenario $s$, \\
$r_j$ &  arrival time of truck $j$,\\
$g_j$ & penalty cost for not serving truck $j$,\\
$\delta_j$ & setup time to align a truck in front of the gate and other preparations, \\
$d_j$ &  latest departure time of truck $j$,\\
$f_i$ &  penalty cost per every unit of time that truck $i$ is waiting idle,\\
$R^p$ & maximum number of available personals, \\
$R^e$ &maximum number of available equipments, \\
$R^v$ & maximum number of available vehicles ,\\
$n_p^s$ & required number of workers in scenario $s$,\\
$n_e^s$ & required number of equipments in scenario $s$,\\
$n_v^s$ & required number of vehicles in scenario $s$,\\
$\mathcal{S}^j$ & set of all possible scenarios for serving truck $j\in\mathcal{J}$.\\
\bottomrule\\

\end{tabular}
}
\end{table}
}

We define the $s$-th \textsc{Resource Deployment Scenario (RDS)} for serving truck $j$ as a quadruplet $s_j^s=(w^s_j, e^s_j, v^s_j, p^s_j)$, where $w_j^s$ stands for the number of workers required to serve truck $j$ under the RDS $s$, $e_j^s$ represents the number of items of equipment required to serve truck $j$ under the RDS $s$, $v_j^s$ indicates the number of vehicles deployed to serve truck $j$ under the RDS $s$ and $p^s_j$ stands for the processing time of truck $j$ under the RDS $s$.\\ 

We define $S=\bigcup_{j\in\mathcal{J}} \mathcal{S}^j$, where $\mathcal{S}^j = \{s^1_j, s^2_j, ..., s_j^{j_{max}}\}$  represents the set of possible resource deployment scenarios for serving truck $j$ (${j_{max}}$ may vary from $1$ to $j$ for different trucks). As the dummy node does not require any resources, only one scenario is associated with it, $s^0=(0,0,0,0)$. { We also define $x_{000}$ to indicate the number of unused docks or in other words the number of docks on which a dummy-to-dummy transition takes place at time '0' and no other real truck being served. } \\ 

$x_{ijts}$ takes $1$, if serving truck $i$ is already finished (the truck has already left) and serving truck $j$ after $i$ after the RDS $s$ has already been started at time period $t$, and $0$ otherwise. $y_{jt}^s$ takes $1$ if truck $j$ is being served in period $t$ and according to the resource allocation scenario $s$, $0$ otherwise. $\eta^s_j$ takes $1$ if truck $j$ is being served under scenario $s$, $0$ otherwise. $z_j$ takes $1$ if truck $j$ could not be served during the day (the planning horizon) and is planned for the next day.\\


\textbf{Dock Assignment and Truck Scheduling Problem with Resource Allocation and Scenario-dependent Processing Time (DATS-RA-SdPT)}\\


If the facility operator is willing to consider RDSs with various processing times, we would call this problem a scenario-dependent variant. In {DATS-RA-SdPT}, we need to know exactly which scenario has been adopted as it affects the processing times. \\

\input{DATS-MSRAMPT}

\textbf{Dock Assignment and Truck Scheduling Problem with Resource Allocation and Scenario-invariant Processing Time (DATS-RA-SiPT)}\\

A special case of {DATS-RA-SdPT} is a conservative variant of the problem, a situation in which the facility operator (the expert) is only willing to consider/propose scenarios with practically the same processing times. Therefore, $\forall s^j_l, s^j_k \in \mathcal{S}^j, l\neq k$, and we have $p_l^{j}=p_k^{j}$. Moreover, $x_{ijts}$ and $p_j^s$ boil down to $x_{ijt}$ and $p_j^s$.

\subsection{Preprocessing, symmetry reduction and valid inequalities}
There are some preprocessing and symmetry elimination efforts that can be formulated as in the sequel. We demonstrate them using the general case and the other one follows analogously.

\subsubsection{Symmetry}
In order to avoid solutions that can be obtained from other solutions with the same objective function’s value, we need to remove the symmetric structures as much as possible.

\begin{corollary}
Dummy truck '0' is a starting point and no RDS really applies to it (no resource is required). Therefore, we fix $s=0$ and can set $ x_{j0ts}= 0, ~\forall j\in\mathcal{J}^\circ, t\in\mathcal{T}, s\in \mathcal{S}^j:s>0$. Otherwise, we could replace $x_{j0ts}$ by  $x_{j0ts'}, s>0$ with the same objective value.
\end{corollary}

\begin{proof}
It naturally follows from the definitions given in \autoref{subsec:params}
\end{proof}

\begin{corollary}
On any dock, the last truck (just before the dummy one, if any) must be discharged as soon as it is served and the dummy truck must start immediately afterwards.
\begin{align}
            & x_{i0ts} \leq \sum_{l} x_{l~i~ t - p_{i}^{s} - \delta_i~s}         & \forall i \in \mathcal{J}, t\in \mathcal{T}, s\in \mathcal{S}^j,\label{eq:MSRAUPT:symmetry}
\end{align}
\end{corollary}
\begin{proof}
Firstly, the last truck on a given dock is not followed with a real truck demanding resources (a real truck not the dummy truck). Secondly, the objective function is only concerned with reducing the waiting time of the real trucks (for those trucks that will receive service, as the dummy truck is not integrated into the objective function). Because of the two aforementioned points, the last truck can postpone its departure without any additional cost on the objective value. We therefore have to make sure that the last truck does not leave arbitrarily and will leave immediately after it becomes possible.
\end{proof}

For DATS-RA-SiPT, it follows analogously.

\subsubsection{Preprocessing and variable fixing}
Given the time windows, the processing time of all trucks and the set of different scenarios, several variables in every model \eqref{eq:MSRAMPT:obj1}- \eqref{eq:MSRAMPT:vars} can be fixed to $0$. Before continuing any further, we first need to introduce the notion of dominance in RDSs.\\

\begin{definition}
For serving a given truck $j\in\mathcal{J}$ two resource deployment scenarios $s_j^1=(w_j^1, e_j^1, v_j^1, p^1_j)$ and $s_j^2=(w_j^2, e_j^2, v_j^2, p^2_j)$ are given. It is said that $s^1$ \emph{generally dominates} $s^2$, $s^1\preceq s^2$, if for $w_j^1 \leq w_j^2, e_j^1 \leq e_j^2$ and $v_j^1 \leq v_j^2$, we obtain $p_j^1\leq p_j^2$.
\begin{itemize}
  \item The dominance is said to be \emph{time-wise strict dominance} ($\prec_T$) if in general dominance we obtain $p_j^1 < p_j^2$.
  \item The dominance is said to be \emph{resource-wise strict dominance} ($\prec_R$) if in general dominance by having strict inequality in at least one of the $w_j^1 \leq w_j^2, e_j^1 \leq e_j^2$ and $v_j^1 \leq v_j^2$, we obtain $p_j^1 < p_j^2$.
  \item The dominance is said to be \emph{strict dominance} ($\prec$) if in general dominance, by having strict inequality in at least one of $w_j^1 \leq w_j^2, e_j^1 \leq e_j^2, v_j^1 \leq v_j^2$ we obtain $p_j^1< p_j^2$.
\end{itemize}
\end{definition}

Clearly, we do not expect to have in our test-bed, the situations wherein $p_j^1 > p_j^2$ while $w_j^1 \leq w_j^2, e_j^1 \leq e_j^2$ and $v_j^1 \leq v_j^2$.

\begin{lem}
For a given truck $j\in\mathcal{J}$ and every $s^1_j, s^2_j\in \mathcal{S}^j$, if $s^1_j$ is in any case strictly dominated by $s^2_j$, then we can assume that $\eta_{j}^{s^1}=0, y_{jt}^{s^1}=0$ and $x_{ijt~s^1}=0$, for all $i\in \mathcal{J}^0, t\in\mathcal{T}$.
\end{lem}

\begin{proof}
According to \eqref{eq:MSRAMPT:eq82}, if truck $j$ is served then there is only one $s$ for which $\eta_{jt}^{s}$ is non zero and according to \eqref{eq:MSRAMPT:eq72} only for that given $s$, we can have $\sum_{t\in \mathcal{T}:r_j+\delta_j<t\leq d_j}\sum_{s\in\mathcal{S}^j} y_{jt}^s=p_j$. Let us assume that $s = s^1_j$. Now, $s^1_j$ for serving truck $j$ is dominated by $s^2_j$ that consumes no more resources but has a strictly lower processing time or fewer resources and does not deliver a worse objective function. Therefore, $s^1_j$ becomes redundant in the presence of $s^2_j$ and can be totally removed from the model, thus reducing the symmetry. Moreover, in \eqref{eq:MSRAMPT:eq71}, $x_{ijt~s^1}=0,~ \forall i\in\mathcal{J}^0$.

\end{proof}

\begin{lem}
In the following situations, variables $x_{ijt}$ and $x_{ijts}$ can be fixed to zero (and be removed from the problem):
1) if $i\geq0,~ t>d_j$, 2) if $i\geq0,~ t< r_j$, 3) if $i\geq0,~ t< r_i+\delta_i+p_i$ in {DATS-RA-SiPT} and $i>0,~ t< r_i+\delta_i+p_{i}^{s}$ in {DATS-RA-SdPT}, 4) if $i\geq 0,~  t+ \delta_j+p_j > d_j$ in {DATS-RA-SiPT} and $i\geq0,~ t+ \delta_j+p_{j}^{s} > d_j$ in {DATS-RA-SdPT}, and 5) if $i=0, t + p_j + \delta_j > |T| -1$ in {DATS-RA-SiPT} and $i=0, t + p_{j}^{s} + \delta_j > |T| -1$  in {DATS-RA-SdPT}.
\end{lem}

\begin{proof}
1)Service to a truck cannot be terminated after its latest departure time, 2)serving truck $j$ cannot be scheduled to start before its earliest arrival time, 3)serving truck $j$ cannot be started before all steps of the preceding trucks have been completed, 4)at time $t$, we cannot start serving truck $j$, if the completion time for it goes beyond the latest departure time and, 5)serving the dummy truck cannot be terminated beyond the planning horizon.

\end{proof}

\begin{lem}
$y_{jt}^s = 0$ for any $t$ not within  $( r_j + \delta_j , d_j ]$.
\end{lem}

\begin{proof}
Since $r_j$ is the earliest arrival time and $\delta_j$ is the time required for $j$ to be set up, ready to be served, truck $j$ cannot be served either any earlier than $r_j+ \delta_j$ or, of course, any later than its latest departure time.
\end{proof}
The use of the aforementioned preprocessing and symmetry elimination techniques has been shown to be very effective in a wide range of initial computational experiments.

\subsubsection{Valid inequalities}
Some classes of valid inequalities can also be identified.
\paragraph{Combinatorial cuts}
Let $\Pi=\{\pi_1, \dots,\pi_{\Pi}\}$ be the set of trucks for which  $t\in\mathcal{T}$ belongs to the feasible time window. For every resource $r\in \mathcal{R}$, we pick a subset of the scenarios available for all trucks (one per truck), say $ \mathcal{S}'\subseteq \mathcal{S}^\pi, \pi\in\Pi$,  for which the use of resource $r$ is violated. The following set of constraints are valid inequalities for the model:

\begin{align}
&   \sum_{j\in {\Pi}}\sum_{s\in \mathcal{S}'\subset \mathcal{S}^j}y_{jt}^s \leq |\Pi| -1, &   \forall t \in \mathcal{T}.
\end{align}
\begin{proof}
This is complementary to the constraints \eqref{eq:MSRAMPT:eq91}-\eqref{eq:MSRAMPT:eq93}.
At every given time $t$ the total number of resources allocated to different trucks must not be violated. Therefore, if any such subset of trucks is identified, at least one truck cannot be served at that given $t$ using one of the proposed RDSs.
\end{proof}

%% file: DATS-MSRAMPT.tex
\begin{align}
            min \:      & \sum_{i,j\in \mathcal{J},t\in \mathcal{T}}  \sum_{s\in\mathcal{S}^j}x_{ijts} (t-r_j)f_j + \sum_{j \in \mathcal{J}} g_j z_j
            \label{eq:MSRAMPT:obj1}\\
            s.\:t.\:\nonumber\\
            &   \sum_{j\in \mathcal{J}}\sum_{t\in \mathcal{T}} \sum_{s\in \mathcal{S}^j}x_{0jts}+ x_{000}  = |\mathcal{D}|,     &      \label{eq:MSRAMPT:eq1}\\
            &   \sum_{i\in \mathcal{J}}\sum_{t\in \mathcal{T}} \sum_{s\in \mathcal{S}^j} x_{i0ts}+ x_{000} = |\mathcal{D}|,     &      \label{eq:MSRAMPT:eq2}\\
            &   \sum_{i\in \mathcal{J}:j\neq i}\sum_{t\in \mathcal{T}} \sum_{s\in \mathcal{S}^j} x_{ijts} = \sum_{s\in \mathcal{S}^j} \eta^s_j,   &   j\in\mathcal{J},\label{eq:MSRAMPT:eq3}\\
            &  \sum_{i\in \mathcal{J}:j\neq i}\sum_{t\in\mathcal{T}} \sum_{s\in \mathcal{S}^j} x_{jits} = \sum_{s\in \mathcal{S}^j} \eta^s_j,   &    j\in\mathcal{J},\label{eq:MSRAMPT:eq4}\\
            %
            %
            & x_{ijts} \leq \sum_{t'\geq \max\{r_l,t+p_{j}^{s}+\delta_j\}} \sum_{l\in \mathcal{J}^\circ:l\neq i , l\neq j}x_{jlt's}         & \forall i\in \mathcal{J}^\circ, j \in \mathcal{J}, t\in \mathcal{T}, s\in \mathcal{S}^j: j\neq i ,\label{eq:MSRAMPT:eq61}\\
            & x_{ijts} \leq y_{jt'}^s    &   \forall  i\in \mathcal{J}^\circ, j \in \mathcal{J}, t\in \mathcal{T}, s\in \mathcal{S}^j, \nonumber\\
            &   & t'\in\{t+\delta_j+1,\dots, t+p_{j}^{s}+\delta_j\}: j\neq i,\label{eq:MSRAMPT:eq71}\\
            & \sum_{t\in \mathcal{T}:r_j+\delta_j<t\leq d_j} y_{jt}^s = p_{j}^{s}\eta^s_j,& \forall j \in \mathcal{J}, s\in \mathcal{S}^j,\label{eq:MSRAMPT:eq72}\\
            &y_{jt}^s  \leq \eta^s_j, &  \forall j \in \mathcal{J}, t\in \mathcal{T}, s\in \mathcal{S}^j,\label{eq:MSRAMPT:eq81}\\
            &\sum_{s\in \mathcal{S^j}} \eta^s_j +z_j = 1, &  \forall j \in \mathcal{J},\label{eq:MSRAMPT:eq82}\\
            &\sum_{ j \in \mathcal{J}:  r_j + \delta_j < t \leq d_j  }\sum_{s\in \mathcal{S}^j}y_{jt}^s  n_p^s\leq R^p, & \forall t\in \mathcal{T}, \label{eq:MSRAMPT:eq91}\\
            &\sum_{ j \in \mathcal{J}:  r_j + \delta_j < t \leq d_j }\sum_{s\in \mathcal{S}^j}y_{jt}^s n_e^s \leq R^e, & \forall t\in \mathcal{T}, \label{eq:MSRAMPT:eq92}\\
            &\sum_{ j \in \mathcal{J}:  r_j + \delta_j < t \leq d_j }\sum_{s\in \mathcal{S}^j}y_{jt}^s  n_v^s \leq R^v, & \forall t\in \mathcal{T},\label{eq:MSRAMPT:eq93}\\
            &  x \in \{0,1\}^ {|\mathcal{J}^\circ||\mathcal{J}^\circ||\mathcal{T}||\mathcal{S}|}, z  \in \{0,1\}^ {|\mathcal{J}|}, x_{000}  \in \{1, ..., |\mathcal{D}|\}, \nonumber\\
            & \eta \in  \{0,1\}^ {|\mathcal{J}||\mathcal{S}|}, y \in  \{0,1\}^ {|\mathcal{J}||\mathcal{S}||\mathcal{T}|}
            .\label{eq:MSRAMPT:vars}
\end{align}

The objective function \eqref{eq:MSRAMPT:obj1} accounts for minimising the cost of trucks waiting to be served and the penalty cost of unfulfilled services, considering different scenarios.
Constraints \eqref{eq:MSRAMPT:eq1} and \eqref{eq:MSRAMPT:eq2} are to make sure that after (before) the dummy starts (ends), $\mathcal{D}$ trucks (dummy truck included) are being served on every dock. Constraints \eqref{eq:MSRAMPT:eq3} (\eqref{eq:MSRAMPT:eq4}) ensure that truck $j$ will precede (follow) another truck $i\in\mathcal{J}^\circ$ at some point in time $t$ on one of the available docks, only if a resource deployment scenario is considered for it; otherwise it will be considered as an unserved truck.
Constraints \eqref{eq:MSRAMPT:eq61} guarantee that at a given period $t$, the truck following truck $j$ on the same dock must give enough time for the docking and processing time of truck $j$, if the time window allows. These constraints avoid loop-backs in feasible solutions. Constraints \eqref{eq:MSRAMPT:eq71} ensure that truck $j$ receives services following a specific RDS. Constraints \eqref{eq:MSRAMPT:eq72} and \eqref{eq:MSRAMPT:eq81} guarantee that if truck $j$ is served, then for $p_j^s$ intervals of time in a row, the resources in a scenario are engaged to fulfill the service for truck $j$. Constraints \eqref{eq:MSRAMPT:eq82} make sure that only one resource allocation scenario can be used for serving a truck. Constraints \eqref{eq:MSRAMPT:eq91}- \eqref{eq:MSRAMPT:eq93} are here to ensure that, under any RDS, the resource capacity is respected at every point in time and for all resource types.\\

%% file: solutionMethod.tex
\section{Solution Methods}\label{sec::solution}
In the phase of modelling and validations, we have conducted extensive computational experiments with many different mathematical models that we have developed, ranging from big-M formulations and compact formulations (for a direct application of general-purpose solvers such as CPLEX) to models with an exponential number of constraints and variables. Our computational experiments confirmed that none of those models and their associated solution processes could scale well to anything close to a real-life size instance of this problem. The main issue with the considered big-M models, as always, was the poor bound improvement and the well-known numerical instability, which rendered them inefficient even for toy-size instances. Compact models with no particular exploitable structures were not efficient as the bound improvement was an issue and the choice of variable to branch on was not done very efficiently by solvers such as CPLEX, Gurobi or those in the COIN-OR ecosystem, as the attributes associated with those variables were very similar, leading to many branches stalling in the same solution quality (we did not identify any remedy either). Similar issues were observed using the models with an exponential number of constraints. For models with an exponential number of columns (path-based models) the pricing problem and its constant changes implied by branching decisions on the master problem were an issue for some mid-size instances.\\

We therefore opted for a model structure that is a good candidate for a Dantzig-Wolfe reformulation, leading to 1) a master problem with properties inherited from the set-covering problem, 2) a branching strategy that can be applied with the original variables and in the master problem, and 3) the pricing problem that can be solved very efficiently using any general-purpose MIP solver.\\

We propose two Dantzig-Wolfe reformulations, one for the model in \eqref{eq:MSRAMPT:obj1}- \eqref{eq:MSRAMPT:vars} and another for the variant with more conservative conditions (which can be considered as a special case of the model proposed in \eqref{eq:MSRAMPT:obj1}- \eqref{eq:MSRAMPT:vars}). Then a branch-and-price approach is proposed to solve them.\black{} 

%

\subsection{Dantzig-Wolfe for DATS-RA-SdPT}
\input{DW-DATS-MSRAMPT}

One may argue that, in fact, $h_j=1-z_j, ~~ \forall j \in \mathcal{J}$, yet it should be emphasised that out of our extensive initial computational experiments a significantly improved computational performance is observed in the presence of $h_j$ instead of its mathematically equivalent $1-z_j$ as revealed empirically.\\


The D-W decomposition for {DATS-RA-SiPT} follows in the same fashion except that we have $x_{ijt}$ instead of $x_{ijts}$.

\subsubsection{Branch-and-price algorithm}\label{sec:cut}

We now elaborate on the main components of our branch-and-price framework.

\paragraph{Cutting plane for pricing problem}

One of the main issues that the pricing problem in many cases will encounter is the presence of \emph{tri-cycles} including dummy trucks. Such unwanted solutions can be eliminated by on-the-fly addition of the following constraints through a branch-and-cut process. Such constraints for {DATS-RA-SiPT} follow:
\begin{align}
&   \sum_{t\in\mathcal{T}}(x_{ijt} + x_{i0t} + x_{j0t}) \leq 2, &   \forall i,j\neq i\in \mathcal{J} \label{const:tricycle1}\\
&   \sum_{t\in\mathcal{T}}(x_{ijt} + x_{0it} + x_{0jt}) \leq 2, &   \forall i,j\neq i\in \mathcal{J} \label{const:tricycle2}
\end{align}

For {DATS-RA-SdPT}, it follows analogously.\textsc{}

%

\paragraph{Tree search strategy and branching}\

Our extensive initial computational experiments have shown that a breadth-first search helps in faster overall convergence.

\paragraph{Branching} \

Constraints  \eqref{eq:MSRAMPT:rmp:eq4} are added to the corresponding master problems to enable branching on the original variables.

\paragraph{Initial column for the master problem}\

For both problems, a trivial solution is not a good feasible solution, wherein none of the trucks will be served and therefore there are $\mathcal{D}$ arcs encompassed from the first dummy to the last dummy. A better solution can be achieved by allocating one truck at a time to the selected docks in the chronological order of their arrival, keeping an eye on the feasibility of the capacity constraints, and backtracking in allocation as soon as infeasibility has been identified. Our initial computational experiments confirm that this often gives a relatively good initial solution to warm-start a branch-and-price procedure. We emphasise that our aim is to keep this procedure very inexpensive, so no improvement is intended at this stage.

\paragraph{The overall algorithm}
The overall flow of the solution process is presented in \autoref{algo}. In line 2, an initial column is being generated. In line 3, the master problem is set up and initialised with the initial column and in line 4, the pricing problem constraints are set up. Lines 5-12, are the main body of the iterations. In line 6, we solve the MP as a MIP problem, which already includes the original variables, and therefore the branching is done on those original variables. In line 7, we use the dual values from line 6 to establish the objective function and solve the pricing problem. Some callbacks are implemented here to separate (\eqref{const:tricycle1}-\eqref{const:tricycle2}) and cut off the irrelevant solutions (invalid pseudo-schedules). In line with this, if the objective function is negative, an improving column can be created from such a solution, added to the MP as a column (a new variable) and then we go back to line 6 to resolve the MP, otherwise we go to 11 while the whole algorithm converges to optimality.

\makeatletter
\def\BState{\State\hskip-\ALG@thistlm}
\makeatother

\begin{algorithm}
\caption{Solution Method.}\label{algo}
\begin{algorithmic}[1]
\Procedure{Branch-and-Price}{}
\State $\textit{init\_col} \gets fillDockUntilResCapViolated()$ \label{alg:initcol}
\State $MP=initMasterProblem(init\_col)$ \label{alg:MP}
\State $PP=initPricingProblemConstraints()$ \label{alg:PP}
\BState \emph{loop}:
\State $solve$ $MP$
\State objPP, PPsol = solve $PP$ by branch-and-cut and separating (\eqref{const:tricycle1}-\eqref{const:tricycle2}), and get pseudo schedules (columns)
    \If {objPP<0}
        \State $MP=MP + col(PPsol)$.
    \Else
        \State break
    \EndIf
\EndProcedure
\end{algorithmic}
\end{algorithm}

%% file: DW-DATS-MSRAMPT.tex
Let $\mathcal{K}$ be the set of pseudo-schedules generated by a pricing sub-problem. We need to introduce variables $h_j, \forall j\in \mathcal{J}^\circ$, which takes 1 if truck $j$ is being served in the proposed solution. Our Dantzig-Wolfe reformulation for {DATS-RA-SdPT} follows.\\

The Restricted Master Problem (RMP) follows:\

\begin{align}
min \:      &  \sum_{t\in \mathcal{T}}\sum_{i,j \in \mathcal{J}^\circ} \sum_{s\in\mathcal{S}^j} \sum_{k=1}^{|\mathcal{K}|}\lambda_{k}f_j(t-r_j)\widehat{x}_{ijts}^{k} \label{eq:MSRAMPT:rmp:obj}\\
            s.\:t.\:\nonumber\\
            (u^1_j):~~~~&   \sum_{r\in \mathcal{R}}\sum_{i\in \mathcal{J}^\circ:j\neq i}\sum_{s\in\mathcal{S}^j}\sum_{t=r_{i}}^{|{\mathcal{T}}|}\left(\sum_{k=1}^{|\mathcal{K}|}\lambda_{k}\widehat{x}_{ijts}^{k}\right) \leq 1,   &   \forall ~ j\in\mathcal{J},\label{eq:MSRAMPT:rmp:eq3}\\
            (u^2_j):~~~~&   \sum_{r\in \mathcal{R}}\sum_{i\in \mathcal{J}^\circ:j\neq i}\sum_{s\in\mathcal{S}^j}\sum_{t=r_{j}}^{|{\mathcal{T}}|}\left(\sum_{k=1}^{|\mathcal{K}|}\lambda_{k}\widehat{x}_{jits}^{k}\right) \leq 1,   &   \forall ~ j\in\mathcal{J},\label{eq:MSRAMPT:MP:eq4}\\
            (\alpha):~~~~&   \sum_{k=1}^{|\mathcal{K}|}\lambda_{k} =1,   &   \label{eq:MSRAMPT:rmp:eq3}                                \\
            (v_{ijts}):~~~~& x_{ijts} - \sum_{k=1}^{|\mathcal{K}|}\lambda_{k}\widehat{x}_{ijts}^{k} = 0 &  \forall  i\in \mathcal{J}^\circ, j \in \mathcal{J}, s\in\mathcal{S}^j, t\in \mathcal{T}\label{eq:MSRAMPT:rmp:eq4}\\
            &   \lambda_k  \geq 0,  & \forall k\in \mathcal{K} ,   \label{eq:MSRAMPT:rmp:eq5}\\
            & x \in \{0,1\}^ {|\mathcal{J}^\circ||\mathcal{J}^\circ||\mathcal{T}||\mathcal{S}|}\label{eq:MSRAMPT:rmp:vars}.
\end{align}

In this case, the Pricing Problem (PP) turns out to be:

\begin{align}
min \:      & \sum_{i,j\in \mathcal{J},t\in \mathcal{T}, s\in\mathcal{S}} f_j(t - r_j)  x_{ijts} + \sum_{j \in \mathcal{J}} g_j (1-\sum_{s\in \mathcal{S}^j} \eta^s_j) \nonumber\\
& - \sum_j u^1_j \left( \sum_{i\in \mathcal{J}^\circ :j\neq i}\sum_{t\in \mathcal{T}}\sum_{s\in\mathcal{S}^j}x_{ijts} \right) \nonumber\\
& - \sum_j u^2_j \left( \sum_{i\in \mathcal{J}^\circ :j\neq i}\sum_{t\in\mathcal{T}}\sum_{s\in \mathcal{S}^j}x_{jits} \right)\nonumber\\
&- \sum_{i,j\in \mathcal{J},t\in \mathcal{T},s\in \mathcal{S}^j} v_{ijts} x_{ijts} -\alpha
\label{eq:MSRAMPT:PP:obj1}\\
            s.\:t.\:\nonumber\\
      &   \sum_{j\in \mathcal{J}}\sum_{t\in \mathcal{T}}\sum_{s\in \mathcal{S}^j} x_{0jts} \leq |\mathcal{D}|,     &  \label{eq:MSRAMPT:PP:eq1}\\
    &   \sum_{i\in \mathcal{J}}\sum_{t\in \mathcal{T}}\sum_{s\in \mathcal{S}^i} x_{i0ts} \leq  |\mathcal{D}|,           &      \label{eq:MSRAMPT:PP:eq2}\\
    &   \sum_{i\in \mathcal{J}^\circ}\sum_{t\in \mathcal{T}} \sum_{s\in \mathcal{S}^j}(x_{ijts} -x_{jits} )= 0,           &   \forall j\in \mathcal{J}   \label{eq:MSRAMPT:PP:eq3}\\
    &   -\sum_{i,j\in \mathcal{J}:j\neq i}\sum_{t\in \mathcal{T}}\sum_{s\in \mathcal{S}^j} x_{ijts} + \sum_{i\in \mathcal{J}} h_i= |\mathcal{D}|,           &     \label{eq:MSRAMPT:PP:eq4}\\
    &   \sum_{t\in\mathcal{T}}\sum_{s\in \mathcal{S}^j} x_{ijts} \leq h_j,           &     \forall i,j\in \mathcal{J}^\circ,\label{eq:MSRAMPT:PP:eq4}\\
    &   \sum_{j\in \mathcal{J}^\circ}\sum_{t\in \mathcal{T}}\sum_{s\in \mathcal{S}^j} x_{ijts} \geq h_i,           &     \forall i\in \mathcal{J}^\circ,\label{eq:MSRAMPT:PP:eq5}\\
            & x_{ijts} \leq \sum_{t'\geq \max\{r_l, r_{j}+p_{j}^{s}+\delta_{j},t+p_{j}^{s}+\delta_j\}} \sum_{l\in \mathcal{J}^\circ:l\neq i , l\neq j}x_{jlt's}    &  i\in \mathcal{J}^\circ,  j \in \mathcal{J}, t\in \mathcal{T},  s\in \mathcal{S}^j: j\neq i ,\label{eq:MSRAMPT:PP:eq61}\\
            & x_{ijts} \leq y_{jt'}^s    &   \forall   i\in \mathcal{J}^\circ, j \in \mathcal{J}, t\in \mathcal{T}, s\in \mathcal{S}^j, \nonumber\\
            &   & t'\in\{t+\delta_j+1,\dots, t+p_{j}^{s}+\delta_j\}: j\neq i,\label{eq:MSRAMPT:PP:eq71}\\
            & \sum_{t\in \mathcal{T}:r_j+\delta_j<t\leq d_j} y_{jt}^s = p_{j}^{s}\eta_{j}^s ,& \forall j \in \mathcal{J}, s\in \mathcal{S}^j,\label{eq:MSRAMPT:PP:eq72}\\
            &y_{jt}^s  \leq \eta^s_j, &  \forall j \in \mathcal{J}, t\in \mathcal{T}, s\in \mathcal{S}^j,\label{eq:MSRAMPT:PP:eq81}\\
            %
            & \eqref{eq:MSRAMPT:eq91}, \eqref{eq:MSRAMPT:eq92}, \eqref{eq:MSRAMPT:eq93}\nonumber,\\
            &\eqref{eq:MSRAMPT:eq61}, \eqref{eq:MSRAMPT:rmp:vars}, \eqref{eq:MSRAUPT:symmetry}\nonumber,\\
            & h  \in \{0,1\}^ {|\mathcal{J}|},  \eta \in  \{0,1\}^ {|\mathcal{J}||\mathcal{S}|}, y \in  \{0,1\}^ {|\mathcal{J}||\mathcal{S}||\mathcal{T}|}.
\end{align}

The description of constraints follows analogously.
%
%
%
%
%

%% file: numerical.tex
\section{Computational Experiments}\label{sec::numerical}

The branch-and-price algorithm is written in C++ and is an in-house engineered code that has been maintained for over a decade. Our computational experiments are conducted on a Dell PRECISION 5820 (Intel Xenon W 2245 (8 cores, 3.9GHz, 4.7GHz Turbo)) and 32GB RAM running on Ubuntu 18.04. The compact formulation \eqref{eq:MSRAMPT:obj1} - \eqref{eq:MSRAMPT:vars} as well as both pricing problems are also solved using CPLEX 12.10.0 Concert C++ interface  while some callbacks such as \texttt{ILOUSERCUTCALLBACK}, \texttt{ILOLAZYCONSTRAINTCALLBACK} ({ in order to identify the violated cuts and  separated them}) and \texttt{ILOMIPINFORCALLBACK} {(in order to query CPLEX for different information such as best bound and incumbent objective value, and many others, along the branch-and-cut process)}, are also implemented to control, customise and influence the flow of the algorithm as well as introducing custom termination criteria.\\

Inspired by the inherent structure of the original real dataset and being restricted by some confidentiality agreements, we had to generate a \emph{skewed} dataset, which is a reasonable representation of the real one.  Rather than relying on a purely random test-bed, we extract the distribution of different events (arrival, departures, service time, etc.) from a historical dataset with 1095 days of operations per cross-dock and use a discrete-event simulation model to simulate the same cross-dock and generate our test-bed. The simulation model is built in \cite{Anylogic} and is a hybrid model of multi-agent simulation and discrete-event simulation. While further elaboration of this model is beyond the scope of this work, in our simulation model, every dock is represented by an agent within which a discrete event model is integrated. The agents communicate with each other, receive a truck, seize and release resources from a pool of lift trucks, labour and other equipment and collect statistics.  \\

\subsection{The current practice}
What is referred to as the '\emph{current practice}' here is described in the following: \\
Currently, only one alternative is available to any scenario, for serving a loading/unloading truck, and this alternative has almost the same processing time. Another observation is that often the number of resources allocated to an inbound/outbound truck varies during the loading/unloading operations, which makes it difficult to calculate a processing time for the amount of resources deployed.\
We carried out an extensive analysis of the data and resorted to several regression and classification techniques to identify and extract, from an enormous volume of historical data, some alternative scenarios and their corresponding processing time. We were then able to generate additional alternative scenarios and estimate their corresponding processing times. A limited number of scenarios were approved (sometimes by some modifications) by the operational managers and some others were dropped (some were judged to be impractical). Conversely, inspired by our proposals, operational managers were coming up with proposals for complete scenarios and we were able to prove or question the associated processing time using the trained machine learning models.\\
We finally concluded, for every truck, a set of scenarios with the same processing times as well as some scenarios with different processing times. A trade-off was made between the number of potential scenarios and the computational intractability.\\

By introducing a portfolio of alternative scenarios and searching among potentially better options, a practical reduction of 5.3 to 11.2 percent in the objective function cost was observed compared to the current practice.
By considering only the first 2 scenarios for serving every truck, one still achieves an objective value that is a lower bound on the objective function's value of the current practice.

\subsection{Instance generation} In the following, we focus only on one cross-dock. The model was tested and validated by carrying out a 1095-day simulation and fine tuning of parameters. We then used this model to generate instances of different sizes. We deal with a regional cross-dock of average size with the number of docks, $20\leq |\mathcal{D}| \leq 60$ and the number of trucks $3\times d  \leq |\mathcal{J}| \leq \min\{4\times d, 200),  \forall d\in \mathcal{D}$. In total, 236 instances were generated. The time unit in our instances is 30 minutes  and an 8-hour shift corresponding to $T=16$ time frames is considered. The arrival times have some similarity to a Poisson distributions function, but we decided instead to randomly generate numbers following a uniform distribution function in the interval $[0, T\times 0.75]$. For the trucks' service and docking time, we used the uniform distribution $U(\frac{T}{8}, \frac{T}{4})$ and $U[1,3]$, respectively. The time windows were generated by augmenting the sum of the aforementioned service and docking times by an additional value from $U[3,5]$\footnote{the grace time is already included.}. Using the historical data and following consultation with the experts for the number of scenarios per truck, we used $U[1,4]$; for the resources, we considered $U(3,6)$ for personnel; $U[1,3]$ for machines and $U[3,7]$ for vehicles. The penalty cost, $f_i$ for every unit of time waiting for truck $i$, based on the expert recommendations, is generated following $U[5,10]$. The penalty cost for not serving a truck is set equal to $g_i = 100\times f_i$, based on the experts' input. \\

Instances are named following the format '\emph{T-xx-d-yy-tr-zz-sce-ss}', where '\emph{xx}' corresponds to $T=16$, '\emph{yy}' represents the number of docks and '\emph{zz}' stands for the number of trucks. Finally, '\emph{ss}' represents the total number of scenarios in an instance.\\
%
%

While using the compact formulation \eqref{eq:MSRAMPT:obj1} - \eqref{eq:MSRAMPT:vars}, we set a time limit \verb"IloCplex:TiLim" to 10,800 seconds in CPLEX. One may observe in the following tables a small deviation from this time limit. Normally, this was due to the fact that CPLEX did not manage to terminate exactly at 10,800 seconds and terminated instead upon completion of the LP resolve at the active node of the branch-and-bound tree.

\subsection{Compact formulations: DATS-RA-SiPT and DATS-RA-SdPT}

In \autoref{tbl:directCplex:SiPT} and \autoref{tbl:directCplex:SdPT}, {}we report our results of applying CPLEX to solve the compact model. \black{}In these tables, the first column corresponds to the instance' name.
The number of nodes processed before termination is reported in ('\#Nodes'), while the status upon termination is reported in ('Status') and the computational time ('CPUTime') for every instance is reported in the last column.\\

            \begin{footnotesize}
                \begin{longtable}{lccc}
                \caption{Computational experiments on  DATS-RA-SiPT with CPLEX.}\label{tbl:directCplex:SiPT} \\
            \specialrule{.2em}{.1em}{.1em}
            \input{SiPT}
                 \end{longtable}
            \end{footnotesize}

            \begin{footnotesize}
                \begin{longtable}{lccc}
                \caption{Computational experiments on DATS-RA-SdPT with CPLEX.}\label{tbl:directCplex:SdPT} \\
            \specialrule{.2em}{.1em}{.1em}
            \input{SdPT}
                 \end{longtable}
            \end{footnotesize}

In both \autoref{tbl:directCplex:SiPT} and \autoref{tbl:directCplex:SdPT}, one observes that the challenge is there even for small size instances with 20 docks and 60-80 trucks within a computational time limit of 3 hours. In the 'Status' column, the occasional parentheses report the relative gap upon meeting the termination criteria, if optimality was not reached. For the few cases in \textbf{bold}, the best feasible solution found was indeed an optimal one but could not be proven within the time limit (it was confirmed by letting the solution process continue beyond the limits defined by the termination criteria). In the 'CPUTime' column, if optimality was not reached within the time limit, we indicate this by '\texttt"tLim"' in the table\\

Every table is divided into three sections corresponding to the instances with 20, 22 and 24 docks. One observes that in \autoref{tbl:directCplex:SiPT} only in the first block were the first few instances solved to optimality within the time limit. In the other two groups, only one of the instances (the smallest one) was solved when CPLEX was being used as a black-box commercial solver. In \autoref{tbl:directCplex:SdPT}, we observe that, as the number of decision variables increases, the performance of CPLEX is significantly lower.\\

One observes that the number of nodes being processed in both tables remains relatively small. This indicates that even solving the LP nodes in the branch-and-bound tree remains a challenge for the solver, hence the smaller number of nodes processed within the time limit.\\

{}While we are not aware of any, there should be models designed to be solver-friendly and perform/scale better on the medium/large size instances. Meanwhile, the model here is proposed to be exploited in a dual decomposition fashion.\black{}

\subsection{Branch-and-price approach}
In \autoref{tbl:realistic}, we report our computational experiments on both DATS-RA-SiPT and DATS-RA-SdPT problem instances. The table is divided into several horizontal sections corresponding to different values of $d\in\mathcal{D}$. Every section has a number of columns, where the first column represents the instance name, the second one reports the number of branch-and-bound nodes processed at the master problem level, the third column indicates how the last pricing problem has been terminated, the fourth one states the computational time, the fifth column reports  the number of columns generated during the process (excluding the initial column) and the last one indicates the number of calls to the pricing problem.\\

For all those instances reported in \autoref{tbl:directCplex:SiPT} and \autoref{tbl:directCplex:SdPT}, we were able to solve the instances in less than 300 seconds. We therefore, do not report them here anymore.\\

We solve the pricing problem using a breadth-first search. The termination criteria are: 1) 10,800 seconds passed and at least one solution with a negative objective value has been observed, 2) the objective value remains at 0 and 6$\times$ 3,600 seconds (6 hours) has been passed on solving the pricing problem. When using CPLEX to solve the pricing problem, we extract all the integer solutions with non-positive objective values encountered along the search and add them as new columns to the master problem. As such solutions remain in the pool, the next call to CPLEX does not start from scratch and the solution process continues from where it left last time. \\

In our initial computational experiments, whenever needed, we separated and added cuts from among the cuts introduced in \autoref{sec:cut}. However, the number of such effective cuts being separated remained extremely limited and the overhead introduced by such separation caused a very low return on investment. We are convinced that this phase can be totally excluded and no cut of this type needs to be separated. \\

We solved 185 instances (corresponding to the realistic size ones) for each of the two studied problems. The proposed branch-and-price solves to optimality 129 instances of DATS-RA-SiPT and 141 instances of DATS-RA-SdPT. In a couple of cases, for DATS-RA-SdPT, the last pricing problem terminated with a gap, but of less than 10 $\%$.\\

One very remarkable observation in both tables is that in no case did we ever need to process more than one node in the restricted master problem. This indicates that the restricted master problem inherits some useful properties from the set-covering problem and tends to deliver integer solutions. This is very useful because every branching changes the structure in the pricing problem, and this becomes challenging in particular if we decide to solve the pricing problem with ad-hoc algorithms. Another aspect is that, if we use a modern MIP solver such as CPLEX to solve the pricing problem, all solutions kept in the solution pool from the previous invocations of CPLEX remain feasible in the new pricing problem and we can start from them as a feasible solution (warm-start).\\

Roughly 70 $\%$ of the instances in the case of DATS-RA-SiPT and 76 $\%$ of the instances in the case of DATS-RA-SdPT were solved to optimality. The sign 'AbortUser' indicates that the overall process was terminated once the user-defined termination criteria were met. We carried out some pathological investigations and it was revealed that a part of this is caused by an excessive level of symmetry from which the pricing problem suffers ---many branchings in the pricing problem with no meaningful impact on the bound improvement. One of the identified causes of this symmetry is the presence of similar scenarios, making it difficult for CPLEX to converge fast. It must be emphasised that this issue becomes less important in DATS-RA-SdPT as this issue with scenarios has been eliminated. \\

We pooled all the solutions encountered by the pricing problem that had a negative objective function value. All these solutions were used to generate columns in the restricted master problem. Therefore, instead of one single column per iteration, we may have many such columns added to the restricted master problem. We also implemented a hashing system to avoid any potential duplication. In any case, the total number of columns added to the pricing problem rarely exceeded 50 (only 3 times).\\

Another remarkable point is that the number of calls to the sub-problem never exceeded 5, while the restricted master problem was always solved at the root node.\\

For calculating the gap, whenever the last pricing problem was not optimal (terminated by 'AbortUser'), the absolute gap between the best incumbent and the best lower (dual) bound is reported.\\

In general, the reported results confirm that we are able to tackle real-size instances of the problem in reasonable time and the approach is computationally viable.
{
\begin{landscape}
\begin{tiny}
    \begin{longtable}{l|ccccc|cccccc}
    \caption{Computational experiments of solving instances of realistic size using the branch-and-price algorithm.}\label{tbl:realistic} \\
\specialrule{.2em}{.1em}{.1em}
\input{Largest}
     \end{longtable}
\end{tiny}
\end{landscape}

}


%% file: SiPT.tex
Instance		& \#Nodes	&Status	&CPUTime (sec.)	\\
\specialrule{.2em}{.1em}{.1em}
tf-16-d-20-tr-60-sce-120&	2536&	Optimal&	6558.14\\
tf-16-d-20-tr-65-sce-120&		256&	Optimal&	2664.80\\
tf-16-d-20-tr-70-sce-120&		5546&	Optimal&	6398.21\\
tf-16-d-20-tr-75-sce-120&		12003&	Optimal&	7806.28\\
tf-16-d-20-tr-80-sce-120&		6725&	\textbf{{\scriptsize Feasible
 (11.32\%)}}&	\texttt{tLim}\\
tf-16-d-20-tr-85-sce-120&		9825&	{\scriptsize Feasible
 (23.25\%)}&	\texttt{tLim}\\
tf-16-d-20-tr-90-sce-120&		3598&	{\scriptsize Feasible
 (41.00\%)}&	\texttt{tLim}\\
tf-16-d-20-tr-95-sce-120&		6988&	{\scriptsize Feasible
 (22.65\%)}&	\texttt{tLim}\\
tf-16-d-20-tr-100-sce-120&		5589&	{\scriptsize Feasible
 (35.55\%)}&	\texttt{tLim}\\
 \hline
tf-16-d-22-tr-66-sce-130&		169&	Optimal&	1824.56\\
tf-16-d-22-tr-71-sce-130&		15009&	\textbf{{\scriptsize Feasible
 (3.14\%)}}&	\texttt{tLim}\\
tf-16-d-22-tr-76-sce-130&		7823&	\textbf{{\scriptsize Feasible
 (4.35\%)}}&	\texttt{tLim}\\
tf-16-d-22-tr-81-sce-130&		9985&	\textbf{{\scriptsize Feasible
 (15.03\%)}}&	\texttt{tLim}\\
tf-16-d-22-tr-86-sce-130&		6689&	{\scriptsize Feasible
 (17.78\%)}&	\texttt{tLim}\\
tf-16-d-22-tr-91-sce-130&		8234&	{\scriptsize Feasible
 (4.69\%)}&	\texttt{tLim}\\
tf-16-d-22-tr-96-sce-130&		6821&	{\scriptsize Feasible
 (25.33\%)}&	\texttt{tLim}\\
tf-16-d-22-tr-101-sce-130&		6529&	{\scriptsize Feasible
 (41.37\%)}&	\texttt{tLim}\\
tf-16-d-22-tr-106-sce-130&		7025&	{\scriptsize Feasible
 (42.25\%)}&	\texttt{tLim}\\
 \hline
tf-16-d-24-tr-72-sce-140&		5889&	Optimal&	8819.08\\
tf-16-d-24-tr-77-sce-140&		2387&	\textbf{{\scriptsize Feasible
 (4.02\%)}}&	\texttt{tLim}\\
tf-16-d-24-tr-82-sce-140&		5435&	{\scriptsize Feasible
 (15.15\%)}&	\texttt{tLim}\\
tf-16-d-24-tr-87-sce-140&		4385&	{\scriptsize Feasible
 (55.26\%)}&	\texttt{tLim}\\
tf-16-d-24-tr-92-sce-140&		6923&	{\scriptsize Feasible
 (44.01\%)}&	\texttt{tLim}\\
tf-16-d-24-tr-97-sce-140&		1786&	{\scriptsize Feasible
 (21.02\%)}&	\texttt{tLim}\\
 tf-16-d-24-tr-102-sce-140&		701&	{\scriptsize Feasible
 (32.44\%)}&	\texttt{tLim}\\
\specialrule{.2em}{.1em}{.1em}

%% file: SdPT.tex
Instance		& \#Nodes	&Status	&CPUTime (sec.)	\\
\specialrule{.2em}{.1em}{.1em}
tf-16-d-20-tr-60-sce-120&	3781&	Optimal&	4692.54\\
tf-16-d-20-tr-65-sce-120&		1012&	Optimal&	5996.25\\
tf-16-d-20-tr-70-sce-120&		3644&	Optimal&	7773.14\\
tf-16-d-20-tr-75-sce-120&		10522&	\textbf{{\scriptsize Feasible
 (11.08\%)}}&	\texttt{tLim}\\
tf-16-d-20-tr-80-sce-120&		8853&	{\scriptsize Feasible
 (20.17\%)}&	\texttt{tLim}\\
tf-16-d-20-tr-85-sce-120&		4215&	\textbf{{\scriptsize Feasible
 (13.39\%)}}&	\texttt{tLim}\\
tf-16-d-20-tr-90-sce-120&		2599&	{\scriptsize Feasible
 (58.21\%)}&	\texttt{tLim}\\
tf-16-d-20-tr-95-sce-120&		6045&	{\scriptsize Feasible
 (24.17\%)}&	\texttt{tLim}\\
tf-16-d-20-tr-100-sce-120&		3699&	{\scriptsize Feasible
 (42.15\%)}&	\texttt{tLim}\\
 \hline
tf-16-d-22-tr-66-sce-130&		1332&	Optimal&	2585.23\\
tf-16-d-22-tr-71-sce-130&		2899&	{\scriptsize Feasible
 (14.22\%)}&	\texttt{tLim}\\
tf-16-d-22-tr-76-sce-130&		6532&	\textbf{{\scriptsize Feasible
 (6.65\%)}}&	\texttt{tLim}\\
tf-16-d-22-tr-81-sce-130&		6855&	{\scriptsize Feasible
 (23.14\%)}&	\texttt{tLim}\\
tf-16-d-22-tr-86-sce-130&		7012&	{\scriptsize Feasible
 (16.88\%)}&	\texttt{tLim}\\
tf-16-d-22-tr-91-sce-130&		6996&	\textbf{{\scriptsize Feasible
 (8.36\%)}}&	\texttt{tLim}\\
tf-16-d-22-tr-96-sce-130&		3458&	{\scriptsize Feasible
 (41.15\%)}&	\texttt{tLim}\\
tf-16-d-22-tr-101-sce-130&		4846&	{\scriptsize Feasible
 (36.38\%)}&	\texttt{tLim}\\
tf-16-d-22-tr-106-sce-130&		6604&	{\scriptsize Feasible
 (39.66\%)}&	\texttt{tLim}\\
 \hline
tf-16-d-24-tr-72-sce-140&		4366&	Optimal&	10015.15\\
tf-16-d-24-tr-77-sce-140&		5801&	\textbf{{\scriptsize Feasible
 (8.28\%)}}&	\texttt{tLim}\\
tf-16-d-24-tr-82-sce-140&		5088&	{\scriptsize Feasible
 (25.23\%)}&	\texttt{tLim}\\
tf-16-d-24-tr-87-sce-140&		2101&	{\scriptsize Feasible
 (48.23\%)}&	\texttt{tLim}\\
tf-16-d-24-tr-92-sce-140&		3256&	{\scriptsize Feasible
 (36.55\%)}&	\texttt{tLim}\\
tf-16-d-24-tr-97-sce-140&		2155&	{\scriptsize Feasible
 (27.25\%)}&	\texttt{tLim}\\
 tf-16-d-24-tr-102-sce-140&		3566&	{\scriptsize Feasible
 (41.15\%)}&	\texttt{tLim}\\
\specialrule{.2em}{.1em}{.1em}

%% file: Largest.tex
&           &           \multicolumn{5}{c}{DATS-RA-SiPT}&           \multicolumn{5}{c}{DATS-RA-SdPT}\\
\hline
Instance	& \#Nodes	&Status	&CPUTime (sec.)	&\#nCols	&\#PrIter	& \#Nodes	&Status	&CPUTime (sec.)	&\#nCols	&\#PrIter		\\
\specialrule{.2em}{.1em}{.1em}
\endfirsthead
\multicolumn{11}{c}{\tablename~\thetable: Computational experiments of solving instances of realistic size using the branch-and-price algorithm. (continued)} \\
\specialrule{.2em}{.1em}{.1em}
&           &           \multicolumn{5}{c}{DATS-RA-SiPT}&           \multicolumn{5}{c}{DATS-RA-SdPT}\\
\hline
Instance	& \#Nodes	&Status	&CPUTime (sec.)	&\#nCols	&\#PrIter	& \#Nodes	&Status	&CPUTime (sec.)	&\#nCols	&\#PrIter		\\
\specialrule{.2em}{.1em}{.1em}
\endhead
tf-16-d-30-tr-90-sce-225	&			1	&	Optimal	&	235.69	&	3	&	2	&	1	&	Optimal	&	286.93	&	20	&	2		\\
tf-16-d-30-tr-95-sce-238	&			1	&	Optimal	&	1055.87	&	9	&	2	&	1	&	Optimal	&	1590.63	&	29	&	2		\\
tf-16-d-30-tr-100-sce-250	&			1	&	Optimal	&	1027.37	&	21	&	2	&	1	&	Optimal	&	928.33	&	26	&	2		\\
tf-16-d-30-tr-105-sce-263	&			1	&	Optimal	&	26.49	&	11	&	2	&	1	&	Optimal	&	2658.41	&	2	&	2		\\
tf-16-d-30-tr-110-sce-275	&			1	&	Optimal	&	4574.31	&	9	&	2	&	1	&	Optimal	&	4681.66	&	17	&	2		\\
tf-16-d-30-tr-115-sce-288	&			1	&	Optimal	&	3599.59	&	15	&	2	&	1	&	Optimal	&	3201.50	&	11	&	2		\\
tf-16-d-30-tr-120-sce-300	&			1	&	Optimal	&	7039.32	&	1	&	2	&	1	&	Optimal	&	4847.27	&	20	&	2		\\
tf-16-d-30-tr-125-sce-313	&			1	&	Optimal	&	640.81	&	21	&	2	&	1	&	Optimal	&	4015.27	&	24	&	2		\\
tf-16-d-30-tr-130-sce-325	&			1	&	Optimal	&	4689.57	&	15	&	2	&	1	&	Optimal	&	2729.48	&	30	&	2		\\
tf-16-d-30-tr-135-sce-338	&			1	&	Optimal	&	8229.29	&	21	&	2	&	1	&	Optimal	&	6099.37	&	23	&	2		\\
tf-16-d-30-tr-140-sce-350	&			1	&	Optimal	&	7044.13	&	22	&	2	&	1	&	Optimal	&	3984.88	&	29	&	2		\\
tf-16-d-30-tr-145-sce-363	&			1	&	Optimal	&	12659.69	&	5	&	2	&	1	&	Optimal	&	10375.50	&	25	&	2		\\
tf-16-d-30-tr-150-sce-375	&			1	&	Optimal	&	3599.84	&	15	&	2	&	1	&	Optimal	&	300.09	&	3	&	2		\\
\hline																											
tf-16-d-32-tr-96-sce-240	&			1	&	Optimal	&	903.43	&	2	&	2	&	1	&	Optimal	&	1890.27	&	1	&	2		\\
tf-16-d-32-tr-101-sce-253	&			1	&	Optimal	&	2291.82	&	7	&	2	&	1	&	Optimal	&	2090.53	&	12	&	2		\\
tf-16-d-32-tr-106-sce-265	&			1	&	Optimal	&	806.01	&	27	&	2	&	1	&	Optimal	&	2695.84	&	11	&	2		\\
tf-16-d-32-tr-111-sce-278	&			1	&	Optimal	&	1737.03	&	31	&	2	&	1	&	Optimal	&	642.22	&	5	&	2		\\
tf-16-d-32-tr-116-sce-290	&			1	&	Optimal	&	6477.01	&	25	&	2	&	1	&	Optimal	&	7038.91	&	8	&	2		\\
tf-16-d-32-tr-121-sce-303	&			1	&	Optimal	&	8092.61	&	13	&	2	&	1	&	Optimal	&	8214.88	&	27	&	2		\\
tf-16-d-32-tr-126-sce-315	&			1	&	Optimal	&	8124.69	&	15	&	2	&	1	&	Optimal	&	11485.39	&	15	&	2		\\
tf-16-d-32-tr-131-sce-328	&			1	&	Optimal	&	4193.22	&	20	&	2	&	1	&	Optimal	&	8437.50	&	32	&	2		\\
tf-16-d-32-tr-136-sce-340	&			1	&	Optimal	&	6433.22	&	20	&	2	&	1	&	Optimal	&	6525.22	&	12	&	2		\\
tf-16-d-32-tr-141-sce-353	&			1	&	Optimal	&	925.51	&	15	&	2	&	1	&	Optimal	&	7927.31	&	20	&	2		\\
tf-16-d-32-tr-146-sce-365	&			1	&	Optimal	&	10285.83	&	28	&	2	&	1	&	Optimal	&	4890.69	&	15	&	2		\\
tf-16-d-32-tr-151-sce-378	&			1	&	Optimal	&	2105.91	&	10	&	2	&	1	&	Optimal	&	12724.49	&	30	&	2		\\
tf-16-d-32-tr-156-sce-390	&			1	&	Optimal	&	2706.79	&	6	&	2	&	1	&	Optimal	&	12701.66	&	11	&	2		\\
\hline																											
tf-16-d-34-tr-102-sce-255	&			1	&	Optimal	&	1202.22	&	15	&	2	&	1	&	Optimal	&	4008.56	&	11	&	2		\\
tf-16-d-34-tr-107-sce-268	&			1	&	Optimal	&	2020.60	&	18	&	2	&	1	&	Optimal	&	5530.82	&	17	&	2		\\
tf-16-d-34-tr-112-sce-280	&			1	&	Optimal	&	4728.14	&	16	&	2	&	1	&	Optimal	&	463.43	&	18	&	2		\\
tf-16-d-34-tr-117-sce-293	&			1	&	Optimal	&	7437.29	&	34	&	2	&	1	&	Optimal	&	6305.11	&	30	&	2		\\
tf-16-d-34-tr-122-sce-305	&			1	&	Optimal	&	7146.63	&	24	&	2	&	1	&	Optimal	&	677.83	&	24	&	2		\\
tf-16-d-34-tr-127-sce-318	&			1	&	Optimal	&	9748.51	&	27	&	2	&	1	&	Optimal	&	917.20	&	11	&	2		\\
tf-16-d-34-tr-132-sce-330	&			1	&	Optimal	&	4429.91	&	13	&	2	&	1	&	Optimal	&	4447.55	&	12	&	2		\\
tf-16-d-34-tr-137-sce-343	&			1	&	Optimal	&	12985.99	&	9	&	2	&	1	&	Optimal	&	9091.44	&	32	&	2		\\
tf-16-d-34-tr-142-sce-355	&			1	&	Optimal	&	12696.96	&	7	&	2	&	1	&	Optimal	&	9796.06	&	14	&	2		\\
tf-16-d-34-tr-147-sce-368	&			1	&	Optimal	&	13662.28	&	17	&	2	&	1	&	Optimal	&	6461.13	&	20	&	2		\\
tf-16-d-34-tr-152-sce-380	&			1	&	Optimal	&	6573.10	&	32	&	2	&	1	&	Optimal	&	17282.59	&	33	&	2		\\
tf-16-d-34-tr-157-sce-393	&			1	&	Optimal	&	8645.88	&	30	&	2	&	1	&	Optimal	&	1885.27	&	30	&	2		\\
tf-16-d-34-tr-162-sce-405	&			1	&	Optimal	&	9901.46	&	4	&	2	&	1	&	Optimal	&	6690.55	&	3	&	2		\\
tf-16-d-34-tr-167-sce-418	&			1	&	Optimal	&	4907.21	&	4	&	2	&	1	&	Optimal	&	22069.24	&	20	&	2		\\
\hline																											
tf-16-d-36-tr-108-sce-270	&			1	&	Optimal	&	1843.30	&	4	&	2	&	1	&	Optimal	&	131.17	&	18	&	2		\\
tf-16-d-36-tr-113-sce-283	&			1	&	Optimal	&	3443.13	&	21	&	2	&	1	&	Optimal	&	4868.46	&	27	&	2		\\
tf-16-d-36-tr-118-sce-295	&			1	&	Optimal	&	2924.09	&	29	&	2	&	1	&	Optimal	&	3290.90	&	4	&	2		\\
tf-16-d-36-tr-123-sce-308	&			1	&	Optimal	&	5638.39	&	33	&	2	&	1	&	Optimal	&	7921.75	&	30	&	2		\\
tf-16-d-36-tr-128-sce-320	&			1	&	Optimal	&	4467.54	&	20	&	2	&	1	&	Optimal	&	1862.74	&	32	&	2		\\
tf-16-d-36-tr-133-sce-333	&			1	&	Optimal	&	11563.01	&	19	&	2	&	1	&	Optimal	&	1637.19	&	34	&	2		\\
tf-16-d-36-tr-138-sce-345	&			1	&	Optimal	&	5230.85	&	10	&	2	&	1	&	Optimal	&	4249.40	&	15	&	2		\\
tf-16-d-36-tr-143-sce-358	&			1	&	Optimal &	3810.31	&	23	&	2	&	1	&	{\tiny AbortUser(5.51)}	&	4728.74	&	19	&	2		\\
tf-16-d-36-tr-148-sce-370	&			1	&	Optimal	&	14509.73	&	36	&	2	&	1	&	{\tiny AbortUser(20.11)}	&	9131.49	&	20	&	2		\\
tf-16-d-36-tr-153-sce-383	&			1	&	Optimal	&	11433.84	&	35	&	2	&	1	&	Optimal	&	4521.12	&	33	&	2		\\
tf-16-d-36-tr-158-sce-395	&			1	&	{\tiny AbortUser(30.00)}		&	14233.21	&	3	&	2	&	1	&	Optimal	&	3287.53	&	27	&	2		\\
tf-16-d-36-tr-163-sce-408	&			1	&	{\tiny AbortUser(46.49)}	&	17862.84	&	9	&	2	&	1	&	Optimal	&	1481.21	&	19	&	2		\\
tf-16-d-36-tr-168-sce-420	&			1	&	{\tiny AbortUser(32.71)}	&	646.06	&	20	&	2	&	1	&	Optimal	&	23023.35	&	34	&	2		\\
tf-16-d-36-tr-173-sce-433	&			1	&	{\tiny AbortUser(81.71)}	&	4024.20	&	20	&	2	&	1	&	Optimal	&	3279.50	&	27	&	2		\\
tf-16-d-36-tr-178-sce-445	&			1	&	{\tiny AbortUser(103.00)}	&	10320.18	&	2	&	2	&	1	&	{\tiny AbortUser(10.45)}	&	15477.90	&	10	&	2		\\
\hline																											
tf-16-d-38-tr-114-sce-285	&			1	&	Optimal	&	1084.50	&	17	&	2	&	1	&	Optimal	&	618.43	&	32	&	2		\\
tf-16-d-38-tr-119-sce-298	&			1	&	Optimal	&	4619.54	&	8	&	2	&	1	&	Optimal	&	6958.46	&	25	&	2		\\
tf-16-d-38-tr-124-sce-310	&			1	&	Optimal	&	6268.22	&	34	&	2	&	1	&	Optimal	&	13202.01	&	14	&	2		\\
tf-16-d-38-tr-129-sce-323	&			1	&	Optimal	&	1020.40	&	15	&	2	&	1	&	Optimal	&	1985.22	&	23	&	2		\\
tf-16-d-38-tr-134-sce-335	&			1	&	Optimal	&	8261.16	&	26	&	2	&	1	&	Optimal	&	10498.56	&	5	&	2		\\
tf-16-d-38-tr-139-sce-348	&			1	&	Optimal	&	7999.22	&	28	&	2	&	1	&	Optimal	&	9965.19	&	15	&	2		\\
tf-16-d-38-tr-144-sce-360	&			1	&	optimal	&	9154.31	&	35	&	2	&	1	&	Optimal	&	628.71	&	23	&	2		\\
tf-16-d-38-tr-149-sce-373	&			1	&	Optimal	&	14025.53	&	38	&	2	&	1	&	Optimal	&	13424.33	&	8	&	2		\\
tf-16-d-38-tr-154-sce-385	&			1	&	Optimal	&	7670.92	&	25	&	2	&	1	&	Optimal	&	4029.71	&	30	&	2		\\
tf-16-d-38-tr-159-sce-398	&			1	&	optimal	&	1404.93	&	27	&	2	&	1	&	{\tiny AbortUser(14.14)}	&	19463.47	&	30	&	2		\\
tf-16-d-38-tr-164-sce-410	&			1	&	{\tiny AbortUser(136.85)}	&	3738.30	&	31	&	2	&	1	&	Optimal	&	13373.05	&	22	&	2		\\
tf-16-d-38-tr-169-sce-423	&			1	&	{\tiny AbortUser(356.44)}	&	6554.25	&	17	&	2	&	1	&	Optimal	&	18532.83	&	16	&	2		\\
tf-16-d-38-tr-174-sce-435	&			1	&	{\tiny AbortUser(89.25)}	&	9350.64	&	2	&	2	&	1	&	{\tiny AbortUser(135.22)}	&	22729.86	&	28	&	2		\\
tf-16-d-38-tr-179-sce-448	&			1	&	{\tiny AbortUser(75.00)}	&	22113.29	&	22	&	2	&	1	&	{\tiny AbortUser(147.19)}	&	4119.03	&	4	&	2		\\
tf-16-d-38-tr-184-sce-460	&			1	&	{\tiny AbortUser(201.32)}	&	26773.60	&	29	&	2	&	1	&	{\tiny AbortUser(202.69)}	&	2798.95	&	25	&	2		\\
tf-16-d-38-tr-189-sce-473	&			1	&	{\tiny AbortUser(312.73)}	&	27639.51	&	15	&	2	&	1	&	{\tiny AbortUser(312.58)}	&	6627.90	&	10	&	2		\\
\hline																											
tf-16-d-40-tr-120-sce-300	&			1	&	Optimal	&	3115.37	&	26	&	2	&	1	&	Optimal	&	8645.86	&	36	&	2		\\
tf-16-d-40-tr-125-sce-313	&			1	&	Optimal	&	6736.44	&	23	&	2	&	1	&	Optimal	&	11665.65	&	39	&	2		\\
tf-16-d-40-tr-130-sce-325	&			1	&	Optimal	&	2196.47	&	23	&	2	&	1	&	Optimal	&	3009.63	&	4	&	2		\\
tf-16-d-40-tr-135-sce-338	&			1	&	Optimal	&	8274.96	&	22	&	2	&	1	&	Optimal	&	7666.37	&	15	&	2		\\
tf-16-d-40-tr-140-sce-350	&			1	&	Optimal	&	123.72	&	21	&	2	&	1	&	Optimal	&	6001.77	&	20	&	2		\\
tf-16-d-40-tr-145-sce-363	&			1	&	Optimal	&	43.67	&	7	&	2	&	1	&	Optimal	&	16365.93	&	16	&	2		\\
tf-16-d-40-tr-150-sce-375	&			1	&	Optimal	&	11294.45	&	17	&	2	&	1	&	Optimal	&	6016.26	&	16	&	2		\\
tf-16-d-40-tr-155-sce-388	&			1	&	Optimal	&	11244.66	&	10	&	2	&	1	&	Optimal	&	10429.31	&	28	&	2		\\
tf-16-d-40-tr-160-sce-400	&			1	&	Optimal	&	9074.38	&	15	&	2	&	1	&	Optimal	&	852.52	&	17	&	2		\\
tf-16-d-40-tr-165-sce-413	&			1	&	optimal	&	822.97	&	9	&	2	&	1	&	{\tiny AbortUser(296.20)}	&	25878.45	&	36	&	2		\\
tf-16-d-40-tr-170-sce-425	&			1	&	Optimal	&	21315.70	&	35	&	2	&	1	&	Optimal	&	30038.73	&	10	&	2		\\
tf-16-d-40-tr-175-sce-438	&			1	&	{\tiny AbortUser(262.84)}	&	5033.66	&	37	&	2	&	1	&	{\tiny AbortUser(45.22)}	&	31902.08	&	15	&	2		\\
tf-16-d-40-tr-180-sce-450	&			1	&	{\tiny AbortUser(301.77)}	&	24284.51	&	23	&	2	&	1	&	{\tiny AbortUser(298.29)}	&	19591.87	&	34	&	2		\\
tf-16-d-40-tr-185-sce-463	&			1	&	{\tiny AbortUser(254.44)}	&	7660.51	&	5	&	2	&	1	&	Optimal	&	22455.90	&	29	&	2		\\
tf-16-d-40-tr-190-sce-475	&			1	&	{\tiny AbortUser(292.57)}	&	26479.31	&	27	&	2	&	1	&	{\tiny AbortUser(385.52)}	&	14282.93	&	25	&	2		\\
tf-16-d-40-tr-195-sce-488	&			1	&	{\tiny AbortUser(411.05)}	&	2878.06	&	24	&	2	&	1	&	{\tiny AbortUser(87.67)}	&	11458.41	&	4	&	2		\\
\hline																											
tf-16-d-42-tr-126-sce-315	&			1	&	Optimal	&	11803.86	&	10	&	2	&	1	&	Optimal	&	3580.64	&	35	&	2		\\
tf-16-d-42-tr-131-sce-328	&			1	&	Optimal	&	6277.80	&	13	&	2	&	1	&	Optimal	&	5044.50	&	1	&	2		\\
tf-16-d-42-tr-136-sce-340	&			1	&	Optimal	&	9068.88	&	20	&	2	&	1	&	Optimal	&	92.90	&	33	&	2		\\
tf-16-d-42-tr-141-sce-353	&			1	&	Optimal	&	1126.50	&	38	&	2	&	1	&	Optimal	&	13617.55	&	24	&	2		\\
tf-16-d-42-tr-146-sce-365	&			1	&	Optimal	&	11568.16	&	1	&	2	&	1	&	Optimal	&	2136.80	&	22	&	2		\\
tf-16-d-42-tr-151-sce-378	&			1	&	Optimal	&	18897.14	&	5	&	2	&	1	&	Optimal	&	1409.64	&	22	&	2		\\
tf-16-d-42-tr-156-sce-390	&			1	&	Optimal	&	20618.73	&	38	&	2	&	1	&	Optimal	&	6992.40	&	10	&	2		\\
tf-16-d-42-tr-161-sce-403	&			1	&	Optimal	&	4484.73	&	23	&	2	&	1	&	Optimal	&	2262.71	&	1	&	2		\\
tf-16-d-42-tr-166-sce-415	&			1	&	Optimal	&	8460.09	&	26	&	2	&	1	&	Optimal	&	29005.46	&	31	&	2		\\
tf-16-d-42-tr-171-sce-428	&			1	&	Optimal	&	17570.49	&	12	&	2	&	1	&	Optimal	&	29163.98	&	27	&	2		\\
tf-16-d-42-tr-176-sce-440	&			1	&	Optimal	&	17466.13	&	21	&	2	&	1	&	Optimal	&	19011.32	&	24	&	2		\\
tf-16-d-42-tr-181-sce-453	&			1	&	{\tiny AbortUser(85.10)}	&	26548.73	&	42	&	2	&	1	&	{\tiny AbortUser(412.45)}	&	35286.14	&	12	&	2		\\
tf-16-d-42-tr-186-sce-465	&			1	&	{\tiny AbortUser(407.00)}	&	12956.92	&	26	&	2	&	1	&	Optimal	&	13525.68	&	33	&	2		\\
tf-16-d-42-tr-191-sce-478	&			1	&	{\tiny AbortUser(261.71)}	&	32013.48	&	22	&	2	&	1	&	{\tiny AbortUser(1222.15)}	&	35278.78	&	32	&	2		\\
tf-16-d-42-tr-196-sce-490	&			1	&	{\tiny AbortUser(845.40)}	&	18764.52	&	33	&	2	&	1	&	{\tiny AbortUser(487.25)}	&	10121.90	&	12	&	2		\\
\hline																											
tf-16-d-44-tr-132-sce-330	&			1	&	Optimal	&	8108.67	&	5	&	2	&	1	&	Optimal	&	1160.32	&	43	&	2		\\
tf-16-d-44-tr-137-sce-343	&			1	&	Optimal	&	4694.80	&	33	&	2	&	1	&	Optimal	&	11217.85	&	12	&	2		\\
tf-16-d-44-tr-142-sce-355	&			1	&	Optimal	&	15019.42	&	43	&	2	&	1	&	Optimal	&	7011.17	&	10	&	2		\\
tf-16-d-44-tr-147-sce-368	&			1	&	Optimal	&	14412.55	&	40	&	2	&	1	&	Optimal	&	14723.45	&	27	&	2		\\
tf-16-d-44-tr-152-sce-380	&			1	&	Optimal	&	13499.23	&	9	&	2	&	1	&	Optimal	&	26169.53	&	22	&	2		\\
tf-16-d-44-tr-157-sce-393	&			1	&	Optimal	&	11604.41	&	3	&	2	&	1	&	Optimal	&	4402.64	&	7	&	2		\\
tf-16-d-44-tr-162-sce-405	&			1	&	Optimal	&	21407.89	&	38	&	2	&	1	&	Optimal	&	8872.95	&	3	&	2		\\
tf-16-d-44-tr-167-sce-418	&			1	&	Optimal	&	2210.36	&	39	&	2	&	1	&	Optimal	&	14830.15	&	42	&	2		\\
tf-16-d-44-tr-172-sce-430	&			1	&	Optimal	&	21570.85	&	2	&	2	&	1	&	{\tiny AbortUser(58.21)}	&	4426.53	&	31	&	2		\\
tf-16-d-44-tr-177-sce-443	&			1	&	{\tiny AbortUser(148.28)}	&	25397.59	&	2	&	2	&	1	&	Optimal	&	10027.16	&	12	&	2		\\
tf-16-d-44-tr-182-sce-455	&			1	&	{\tiny AbortUser(269.10)}	&	22596.59	&	8	&	2	&	1	&	{\tiny AbortUser(128.36)}	&	3093.20	&	13	&	2		\\
tf-16-d-44-tr-187-sce-468	&			1	&	Optimal	&	28162.16	&	10	&	2	&	1	&	{\tiny AbortUser(148.74)}	&	34484.66	&	15	&	2		\\
tf-16-d-44-tr-192-sce-480	&			1	&	Optimal	&	20977.78	&	32	&	2	&	1	&	{\tiny AbortUser(223.44)}	&	18999.82	&	7	&	2		\\
tf-16-d-44-tr-197-sce-493	&			1	&	{\tiny AbortUser(580.53)}	&	2863.76	&	34	&	2	&	1	&	{\tiny AbortUser(115.25)}	&	29043.71	&	34	&	2		\\
\hline																											
tf-16-d-46-tr-138-sce-345	&			1	&	Optimal	&	17105.27	&	45	&	2	&	1	&	Optimal	&	6391.85	&	28	&	2		\\
tf-16-d-46-tr-143-sce-358	&			1	&	Optimal	&	12635.66	&	20	&	2	&	1	&	Optimal	&	9058.01	&	17	&	2		\\
tf-16-d-46-tr-148-sce-370	&			1	&	Optimal	&	2847.70	&	30	&	2	&	1	&	Optimal	&	23718.00	&	30	&	2		\\
tf-16-d-46-tr-153-sce-383	&			1	&	Optimal	&	22499.03	&	3	&	2	&	1	&	Optimal	&	5126.84	&	16	&	2		\\
tf-16-d-46-tr-158-sce-395	&			1	&	Optimal	&	6398.52	&	38	&	2	&	1	&	Optimal	&	8696.67	&	11	&	2		\\
tf-16-d-46-tr-163-sce-408	&			1	&	Optimal	&	16867.99	&	1	&	2	&	1	&	Optimal	&	25967.47	&	9	&	2		\\
tf-16-d-46-tr-168-sce-420	&			1	&	Optimal	&	2023.58	&	14	&	2	&	1	&	Optimal	&	17830.72	&	28	&	2		\\
tf-16-d-46-tr-173-sce-433	&			1	&	Optimal	&	29507.91	&	32	&	2	&	1	&	Optimal	&	27180.27	&	6	&	2		\\
tf-16-d-46-tr-178-sce-445	&			1	&	{\tiny AbortUser(451.27)}	&	4086.57	&	36	&	2	&	1	&	Optimal	&	13944.00	&	16	&	2		\\
tf-16-d-46-tr-183-sce-458	&			1	&	{\tiny AbortUser(425.42)}	&	9247.53	&	28	&	2	&	1	&	{\tiny AbortUser(185.69)}	&	36295.84	&	41	&	2		\\
tf-16-d-46-tr-188-sce-470	&			1	&	Optimal	&	29290.65	&	1	&	2	&	1	&	{\tiny AbortUser(125.96)}	&	33876.64	&	38	&	2		\\
tf-16-d-46-tr-193-sce-483	&			1	&	{\tiny AbortUser(69.96)}	&	321.12	&	21	&	2	&	1	&	{\tiny AbortUser(482.36)}	&	32665.91	&	13	&	2		\\
tf-16-d-46-tr-198-sce-495	&			1	&	{\tiny AbortUser(1965.16)}	&	32940.72	&	7	&	2	&	1	&	Optimal	&	37661.68	&	44	&	2		\\
\hline																											
tf-16-d-48-tr-144-sce-360	&			1	&	Optimal	&	1946.99	&	25	&	2	&	1	&	Optimal	&	20905.42	&	4	&	2		\\
tf-16-d-48-tr-149-sce-373	&			1	&	Optimal	&	11337.13	&	48	&	2	&	1	&	Optimal	&	14069.85	&	8	&	2		\\
tf-16-d-48-tr-154-sce-385	&			1	&	Optimal	&	9992.61	&	11	&	2	&	1	&	Optimal	&	2583.20	&	1	&	2		\\
tf-16-d-48-tr-159-sce-398	&			1	&	Optimal	&	6890.31	&	21	&	2	&	1	&	Optimal	&	13026.42	&	38	&	2		\\
tf-16-d-48-tr-164-sce-410	&			1	&	Optimal	&	19460.70	&	36	&	2	&	1	&	Optimal	&	15586.51	&	4	&	2		\\
tf-16-d-48-tr-169-sce-423	&			1	&	Optimal	&	21008.31	&	28	&	2	&	1	&	Optimal	&	15888.43	&	7	&	2		\\
tf-16-d-48-tr-174-sce-435	&			1	&	Optimal	&	20569.09	&	32	&	2	&	1	&	{\tiny AbortUser(185.13)}	&	28191.02	&	46	&	2		\\
tf-16-d-48-tr-179-sce-448	&			1	&	Optimal	&	18986.21	&	32	&	2	&	1	&	Optimal	&	4516.18	&	24	&	2		\\
tf-16-d-48-tr-184-sce-460	&			1	&	{\tiny AbortUser(1105.15)}	&	4592.94	&	39	&	2	&	1	&	Optimal	&	8824.93	&	13	&	2		\\
tf-16-d-48-tr-189-sce-473	&			1	&	{\tiny AbortUser(982.11)}	&	14143.83	&	43	&	2	&	1	&	Optimal	&	8440.14	&	3	&	2		\\
tf-16-d-48-tr-194-sce-485	&			1	&	Optimal	&	31735.83	&	16	&	2	&	1	&	Optimal	&	43003.31	&	12	&	2		\\
tf-16-d-48-tr-199-sce-498	&			1	&	{\tiny AbortUser(1693.12)}	&	16053.66	&	45	&	2	&	1	&	Optimal	&	364.84	&	8	&	2		\\
\hline																											
tf-16-d-50-tr-150-sce-375	&			1	&	Optimal	&	22017.75	&	7	&	2	&	1	&	Optimal	&	9612.79	&	25	&	2		\\
tf-16-d-50-tr-155-sce-388	&			1	&	Optimal	&	4541.82	&	29	&	2	&	1	&	Optimal	&	31417.37	&	34	&	2		\\
tf-16-d-50-tr-160-sce-400	&			1	&	Optimal	&	8590.92	&	11	&	2	&	1	&	Optimal	&	11197.69	&	30	&	2		\\
tf-16-d-50-tr-165-sce-413	&			1	&	Optimal	&	7249.05	&	30	&	2	&	1	&	Optimal	&	26552.88	&	34	&	2		\\
tf-16-d-50-tr-170-sce-425	&			1	&	optimal	&	20035.59	&	16	&	2	&	1	&	Optimal	&	15126.80	&	39	&	2		\\
tf-16-d-50-tr-175-sce-438	&			1	&	Optimal	&	4790.98	&	44	&	2	&	1	&	{\tiny AbortUser(178.28)}	&	26979.40	&	35	&	2		\\
tf-16-d-50-tr-180-sce-450	&			1	&	Optimal	&	25194.18	&	29	&	2	&	1	&	Optimal	&	14393.04	&	21	&	2		\\
tf-16-d-50-tr-185-sce-463	&			1	&	{\tiny AbortUser(365.63)}	&	22964.42	&	36	&	2	&	1	&	{\tiny AbortUser(187.27)}	&	28425.11	&	49	&	2		\\
tf-16-d-50-tr-190-sce-475	&			1	&	{\tiny AbortUser(415.87)}	&	13262.91	&	3	&	2	&	1	&	{\tiny AbortUser(109.55)}	&	990.10	&	43	&	2		\\
tf-16-d-50-tr-195-sce-488	&			1	&	{\tiny AbortUser(298.92)}	&	14409.58	&	49	&	2	&	1	&	Optimal	&	20389.97	&	24	&	2		\\
\hline																											
tf-16-d-52-tr-156-sce-390	&			1	&	Optimal	&	14091.38	&	30	&	2	&	1	&	Optimal	&	18522.81	&	27	&	2		\\
tf-16-d-52-tr-161-sce-403	&			1	&	Optimal	&	19920.50	&	26	&	2	&	1	&	Optimal	&	744.69	&	32	&	3		\\
tf-16-d-52-tr-166-sce-415	&			1	&	Optimal	&	20829.76	&	36	&	2	&	1	&	Optimal	&	26321.42	&	38	&	3		\\
tf-16-d-52-tr-171-sce-428	&			1	&	Optimal	&	8010.56	&	22	&	2	&	1	&	Optimal	&	2792.71	&	32	&	2		\\
tf-16-d-52-tr-176-sce-440	&			1	&	{\tiny AbortUser(368.44)}	&	2865.29	&	35	&	2	&	1	&	Optimal	&	43375.39	&	52	&	2		\\
tf-16-d-52-tr-181-sce-453	&			1	&	{\tiny AbortUser(436.80)}	&	15777.11	&	14	&	2	&	1	&	{\tiny AbortUser(189.98)}	&	1982.33	&	12	&	2		\\
tf-16-d-52-tr-186-sce-465	&			1	&	Optimal	&	34072.89	&	45	&	2	&	1	&	{\tiny AbortUser(259.58)}	&	50134.15	&	27	&	2		\\
tf-16-d-52-tr-191-sce-478	&			1	&	{\tiny AbortUser(1250.00)}	&	40094.30	&	34	&	2	&	1	&	Optimal	&	904.25	&	22	&	2		\\
tf-16-d-52-tr-196-sce-490	&			1	&	{\tiny AbortUser(1511.22)}	&	29891.91	&	47	&	2	&	1	&	{\tiny AbortUser(106.12)}	&	18903.94	&	9	&	2		\\
\hline																											
tf-16-d-54-tr-162-sce-405	&			1	&	Optimal	&	8679.24	&	14	&	2	&	1	&	Optimal	&	22029.86	&	41	&	2		\\
tf-16-d-54-tr-167-sce-418	&			1	&	Optimal	&	19568.24	&	23	&	2	&	1	&	Optimal	&	39065.80	&	43	&	2		\\
tf-16-d-54-tr-172-sce-430	&			1	&	Optimal	&	3471.10	&	43	&	2	&	1	&	Optimal	&	12945.98	&	28	&	3		\\
tf-16-d-54-tr-177-sce-443	&			1	&	Optimal	&	10138.81	&	21	&	2	&	1	&	Optimal	&	29548.30	&	6	&	2		\\
tf-16-d-54-tr-182-sce-455	&			1	&	{\tiny AbortUser(450.10)}	&	15202.66	&	25	&	2	&	1	&	Optimal	&	37336.13	&	29	&	2		\\
tf-16-d-54-tr-187-sce-468	&			1	&	{\tiny AbortUser(915.98)}	&	8625.98	&	21	&	2	&	1	&	{\tiny AbortUser(1058.55)}	&	31789.75	&	8	&	3		\\
tf-16-d-54-tr-192-sce-480	&			1	&	{\tiny AbortUser(1610.01)}	&	35787.75	&	40	&	2	&	1	&	{\tiny AbortUser(1147.26)}	&	45322.06	&	32	&	2		\\
tf-16-d-54-tr-197-sce-493	&			1	&	{\tiny AbortUser(1703.16)}	&	1035.94	&	54	&	2	&	1	&	{\tiny AbortUser(1253.78)}	&	46239.66	&	15	&	2		\\
\hline																											
tf-16-d-56-tr-168-sce-420	&			1	&	Optimal	&	5306.28	&	26	&	2	&	1	&	Optimal	&	8899.40	&	10	&	2		\\
tf-16-d-56-tr-173-sce-433	&			1	&	Optimal	&	6929.31	&	35	&	2	&	1	&	Optimal	&	29610.01	&	5	&	2		\\
tf-16-d-56-tr-178-sce-445	&			1	&	{\tiny AbortUser(982.26)}	&	17643.67	&	55	&	2	&	1	&	Optimal	&	6259.07	&	43	&	2		\\
tf-16-d-56-tr-183-sce-458	&			1	&	{\tiny AbortUser(586.14)}	&	11025.66	&	32	&	2	&	1	&	Optimal	&	40246.82	&	18	&	3		\\
tf-16-d-56-tr-188-sce-470	&			1	&	{\tiny AbortUser(658.26)}	&	32797.28	&	55	&	2	&	1	&	{\tiny AbortUser(1140.45)}	&	23771.74	&	30	&	2		\\
tf-16-d-56-tr-193-sce-483	&			1	&	{\tiny AbortUser(615.88)}	&	25706.69	&	36	&	2	&	1	&	{\tiny AbortUser(1852.20)}	&	53734.43	&	48	&	2		\\
tf-16-d-56-tr-198-sce-495	&			1	&	{\tiny AbortUser(825.20)}	&	1103.81	&	5	&	2	&	1	&	{\tiny AbortUser(1066.70)}	&	32597.82	&	38	&	2		\\
\hline																											
tf-16-d-58-tr-174-sce-435	&			1	&	{\tiny AbortUser(1682.57)}	&	15851.48	&	26	&	2	&	1	&	Optimal	&	4527.39	&	1	&	2		\\
tf-16-d-58-tr-179-sce-448	&			1	&	{\tiny AbortUser(1267.87)}	&	13050.69	&	53	&	2	&	1	&	Optimal	&	35781.98	&	12	&	2		\\
tf-16-d-58-tr-184-sce-460	&			1	&	{\tiny AbortUser(1398.22)}	&	2991.38	&	15	&	2	&	1	&	{\tiny AbortUser(1890.19)}	&	39438.16	&	44	&	2		\\
tf-16-d-58-tr-189-sce-473	&			1	&	{\tiny AbortUser(1036.55)}	&	18286.96	&	3	&	2	&	1	&	{\tiny AbortUser(2367.52)}	&	51285.98	&	30	&	3		\\
tf-16-d-58-tr-194-sce-485	&			1	&	{\tiny AbortUser(1448.20)}	&	9061.02	&	11	&	2	&	1	&	{\tiny AbortUser(695.25)}	&	46578.77	&	26	&	3		\\
tf-16-d-58-tr-199-sce-498	&			1	&	{\tiny AbortUser(1599.14)}	&	45260.86	&	7	&	2	&	1	&	{\tiny AbortUser(485.36)}	&	28978.38	&	12	&	2		\\
\hline																											
tf-16-d-60-tr-180-sce-450	&			1	&	{\tiny AbortUser(2501.88)}	&	43139.41	&	32	&	2	&	1	&	{\tiny AbortUser(785.45)}	&	53356.85	&	42	&	2		\\
tf-16-d-60-tr-185-sce-463	&			1	&	{\tiny AbortUser(4150.24}	&	31872.99	&	5	&	2	&	1	&	Optimal	&	9997.00	&	39	&	2		\\
tf-16-d-60-tr-190-sce-475	&			1	&	{\tiny AbortUser(2687.26)}	&	24513.45	&	32	&	2	&	1	&	{\tiny AbortUser(3659.87)}	&	37563.02	&	9	&	3		\\
tf-16-d-60-tr-195-sce-488	&			1	&	{\tiny AbortUser(1452.36)}	&	25809.94	&	14	&	2	&	1	&	{\tiny AbortUser(4460.42)}	&	47546.25	&	34	&	2		\\ 

%% file: conclusion.tex
\section{Summary, conclusion and outlook to future works}\label{sec::conclusion}
Describing the processing time in terms of the resources deployed for carrying out a task (serving a truck), even if it would have been obvious, leads to unnecessary overhead and often some nonlinearity in the models, which in turn will complicate the solution process. In real cases, not all possible/alternative combinations of those resources are practical and they are sometimes even too unrealistic. Moreover, often different scenarios report processing times that are 'almost' the same from the experts' viewpoint, given the granularity of time units. These situations cause efficiency issues on the solution process and such modellings are less often suitable in decision support systems. Given that such a set of practical scenarios is very limited, instead of dealing with scenarios showing only some minor decimal points differences in a closed-form formulation, we opted to establish a set of scenarios based on expert knowledge.\

{}We then proposed a mathematical programming model for the dock assignment and truck scheduling problem. In this study, multiple resource deployment scenarios are considered for every truck in the system. The proposed model is tailored for conservative decision makers and considers different scenarios having different processing times. The proposed model is particularly designed to be exploited in a dual decomposition framework such as Dantzig-Wolfe and branch-and-price. While the compact model do not show very interesting computational behaviour if directly given to CPLEX, the proposed Dantzig-Wolfe reformulation and the branch-and-price approach equipped with a warm start initial column heuristic is shown to be very efficient in solving real-size instances of the problem in reasonable time. \

\black{}Here, due to some hardware limitations, we were only able to deal with a few (say three in this work) resources ---a sort of aggregation of resources of similar type. In reality, these resources can be disaggregated into a number of distinct groups (labour with different skills, vehicles of different types, etc.). We therefore require some even more efficient solution approaches to reduce the computation time and effort by several orders of magnitude, enabling us to integrate more features of the real practices. \

Further research directions include polyhedral analysis, identifying some classes of tightening facet-defining valid inequalities mainly to improve the quality of the proposed model and reduce the integrality gap, and having more efficient solution algorithms for solving the pricing problems. Alternative ad-hoc inspection-based solution approaches to the pricing problem may improve the overall efficiency. Normally, the calls to a cross-dock follow a statistical distribution and those generated instances follow more or less the same random distribution, so approaches based on machine learning and neural networks can be exploited to learn problem instances labelled by optimal solutions and produce some high quality feasible solutions.